\documentclass{article}
\title{Faithfulness of generalised Verma modules for Iwasawa algebras}
\author{Stephen Mann}

\usepackage{amsmath, amsthm, amssymb, amsfonts, mathtools,stmaryrd,hyperref}
\theoremstyle{definition}
\newtheorem{mydef}{Definition}[subsection]
\newtheorem{myex}[mydef]{Example}

\theoremstyle{plain}
\newtheorem{mythm}[mydef]{Theorem}
\newtheorem{mylem}[mydef]{Lemma}
\newtheorem{mycor}[mydef]{Corollary}
\newtheorem{myprop}[mydef]{Proposition}
\newtheorem{Theorem}{Theorem}

\theoremstyle{remark}

\DeclareMathOperator{\Ann}{Ann}

\newcommand{\Ug}{\ensuremath{\widehat{U(\mathfrak{g})_{n,K}}}}

\newcommand{\ug}{\ensuremath{U(\mathfrak{g}_{K})}}

\newcommand{\Up}{\ensuremath{\widehat{U(\mathfrak{p})_{n,K}}}}

\newcommand{\up}{\ensuremath{U(\mathfrak{p}_{K})}}

\newcommand{\Ul}{\ensuremath{\widehat{U(\mathfrak{l})_{n,K}}}}

\newcommand{\ul}{\ensuremath{U(\mathfrak{l}_{K})}}

\newcommand{\Upp}{\ensuremath{\widehat{U(\mathfrak{p'})_{n,K}}}}
\newcommand{\upp}{\ensuremath{U(\mathfrak{p}'_{K})}}

\newcommand{\Un}{\ensuremath{\widehat{U(\mathfrak{n})_{n,K}}}}

\newcommand{\un}{\ensuremath{U(\mathfrak{n}_{K})}}

\newcommand{\Uh}{\ensuremath{\widehat{U(\mathfrak{h})_{n,K}}}}

\newcommand{\uh}{\ensuremath{U(\mathfrak{h}_{K})}}

\newcommand{\Of}{\mathcal{O}_{F}}

\begin{document}

\maketitle

\begin{abstract}
\noindent We prove faithfulness of infinite-dimensional generalised Verma modules for Iwasawa algebras corresponding to split simple Lie algebras with a Chevalley basis. We use this to prove faithfulness of all infinite-dimensional highest-weight modules in the case of type $A_{2}$. In this case we also show that all prime ideals of the corresponding Iwasawa algebras are annihilators of finite-dimensional simple modules.

\end{abstract}

\section{Introduction}

\subsection{Motivation and main results}

Suppose $G$ is a profinite group and $K$ is a complete discretely valued field. Then the Iwasawa algebra $KG$ is a completion of the group algebra $K[G]$. Iwasawa algebras are useful for studying the representation theory of compact p-adic Lie groups since there is a relationship between modules over the Iwasawa algebra and representations of the group, see \cite[Theorem 3.5]{Schneider2}. Iwasawa algebras are also being used to study problems in Number Theory, see for instance \cite{Application1} and \cite{Application2}.

When studying any ring, an important problem is to study the prime spectrum. It is an ongoing project to classify the prime ideals of Iwasawa algebras, see for instance \cite{verma}, \cite{controller}, \cite{nilpotent}. Results so far suggest that for noncommutative Iwasawa algebras the prime ideals are fairly sparse. This paper continues in this direction in the case of certain noncommutative Iwasawa algebras in characteristic $0$. 

This paper makes progress towards answering Question B from \cite{verma}, which if answered positively can be used to show that all non-zero prime ideals of the relevant Iwasawa algebras arise as annihilators of finite-dimensional simple modules. Specifically, we aim to study the Iwasawa algebras of uniform pro-p groups corresponding to simple Lie algebras with a Chevalley basis. However, for our main result it is helpful to make a more general statement for reductive groups for the purpose of an inductive proof.

Let $p$ be an odd prime. Let $F$ be a finite field extension of $\mathbb{Q}_{p}$ with ring of integers $\Of$, and $K$ be a discretely valued field extension of $F$ with ring of integers $R$. Suppose $G$ is an $F$-uniform analytic pro-p group with corresponding Lie algebra $L_{G} = p^{n+1} \mathfrak{g}$, where $\mathfrak{g}$ is an $\Of$-lattice of a split simple $F$-Lie algebra and has a Chevalley basis.

We then have a corresponding Iwasawa algebra $KG$, which embeds in a completion $\Ug$ of the universal enveloping algebra of $\ug$. 
We then can consider a class of $KG$-modules corresponding to highest-weight $\ug$-modules. 

Specifically, we define an affinoid highest weight module for $\Ug$ to be a module of the form $\Ug \otimes_{\ug} M$, where $M$ is a highest-weight module for $\ug$ with a highest weight $\lambda$ satisfying the condition $\lambda(p^{n} \mathfrak{h}_{R}) \subseteq R$. (This condition on $\lambda$ is assumed since $\Ug \otimes_{\ug} M$ will in fact be zero otherwise).

The study of affinoid highest-weight modules has applications for studying the prime ideals of $KG$. In \cite{verma}, two questions are asked.

\medskip

\noindent \emph{Question A}: does every primitive ideal of $\Ug$ with $K$-rational central character arise as the annihilator of a simple affinoid highest-weight module?

\medskip

\noindent \emph{Question B}: is every affinoid highest-weight module that is infinite-dimensional over $K$ faithful as a $KG$-module?

\medskip

\noindent Given positive answers to both of these questions, in the case where $F = \mathbb{Q}_{p}$, $R$ has uniformiser $p$ and $p$ is sufficiently large relative to the root system of $\mathfrak{g}$, we can prove that every non-zero prime ideal of $KG$ arises as the annihilator of a simple finite-dimensional module. This is seen in \S \ref{applicationsection}. Question A has been answered positively in \cite{Ioan}.

 This paper makes progress towards a positive answer to question B. Specifically, it was shown in \cite{verma} that affinoid Verma modules are faithful for $KG$. In this paper, we extend this to generalised Verma modules. 
 
\begin{Theorem}
\label{theorem1}
Suppose $\mathfrak{g}$ is an $\Of$-lattice of a split simple $F$-Lie algebra with root system $\Phi$ such that $\mathfrak{g}$ has a Chevalley basis, and $G$ is an $F$-uniform group with $L_{G} = p^{n+1} \mathfrak{g}$. Assume $p > 2$ is large enough that $p$ does not divide the determinant of the  Cartan matrix of any root subsystem of $\Phi$.  Suppose $M$ is an infinite-dimensional generalised Verma module for $\mathfrak{g}$ with highest-weight $\lambda$ satisfying $\lambda (p^{n}\mathfrak{h}_{R}) \subseteq R$ where $\mathfrak{h}$ is the Cartan subalgebra of $\mathfrak{g}$. Then $\widehat{M} = \Ug \otimes_{\ug} M$ is faithful as a $KG$-module.
\end{Theorem}

\noindent Also, we give a positive answer to question B in the case where $\mathfrak{g} = \mathfrak{sl}_{3}(\Of)$, which we do in \S \ref{sl3case}. In particular, this implies that in the case of $\mathfrak{g} = \mathfrak{sl}_{3}(\mathbb{Q}_{p})$, where $R$ has uniformiser $p$ and for $p>3$, all non-zero prime ideals of $KG$ are annihilators of simple finite-dimensional modules.

\begin{Theorem}
\label{theorem2}
Suppose $p>3$ and $\mathfrak{g} = \mathfrak{sl}_{3}(\Of)$ with Cartan subalgebra $\mathfrak{h}$. Suppose $G$ is an $F$-uniform group with $L_{G} = p^{n+1} \mathfrak{g}$. Suppose $\lambda : \mathfrak{h}_{K} \rightarrow K$ is a weight such that $\lambda (p^{n} \mathfrak{h}_{R}) \subseteq R$ and $M$ is an infinite-dimensional highest-weight module for $\ug$ of highest-weight $\lambda$. Then $\hat{M}$ is faithful as a $KG$-module.
\end{Theorem}

\begin{Theorem}
\label{theorem3}
Suppose $G$ is a uniform pro-p group with $L_{G} = \mathfrak{sl}_{3}(\mathbb{Z}_{p})$ and that $p > 3$. Assume $R$ has uniformiser $p$. Then all prime ideals of $KG$ are annihilators of finite-dimensional simple modules.
\end{Theorem}

\subsection{Acknowledgements}

This paper was written with the help of my supervisor Simon Wadsley.

The inductive structure of the proof of the main theorem (Theorem \ref{mainthm}) comes from the work of Oliver Janzer and Harry Giles from a summer research project with Simon Wadsley, in which they proved the theorem for the special case of generalised Verma modules for $\mathfrak{sl}_{m+1}(\mathbb{Q}_{p})$ with highest-weight $0$.

This paper builds on the work of Konstantin Ardakov and Simon Wadsley in \cite{verma} and \cite{irredreps}.

This project was funded by the Engineering and Physical Sciences Research Council through the Doctoral Partnership Training grant allocation.

\subsection{Set-up}

Throughout, we fix a deformation parameter $n \in \mathbb{N}_{0}$.

Let $\Phi$ be a root system with a fixed set of simple roots $\Delta$ and corresponding positive system $\Phi^{+}$. Let $p$ be an odd prime with the property that $p$ does not divide the determinant of the Cartan matrix of any root subsystem of $\Phi$.

Let $F$ be a finite field extension of $\mathbb{Q}_{p}$ with a ring of integers $\Of$. Let $K$ be a complete discretely valued field extension of $F$ with ring of integers $R$ and a uniformiser $\pi$. 

Let $\mathfrak{g}$ be an $\Of$-lattice of a split reductive $F$-Lie algebra $\mathfrak{g}_{F}$, such that $\mathfrak{g}$ has the form $\mathfrak{g} = \mathfrak{g}_{s} \oplus \mathfrak{a}$, where $\mathfrak{a}$ is the centre of $\mathfrak{g}$ and $\mathfrak{g}_{s,F}$ is split semisimple. Assume $\mathfrak{g}_{s}$ has a Chevalley basis $\{e_{\alpha},f_{\alpha} \mid \alpha \in \Phi^{+} \} \cup \{h_{\alpha} \mid \alpha \in \Delta \}$ and Cartan subalgebra $\mathfrak{h}_{s}$ spanned by the $h_{\alpha}$. Write $\mathfrak{h} = \mathfrak{h}_{s} \oplus \mathfrak{a}$.

Let $G$ be an $F$-uniform group with associated $\Of$-Lie algebra $L_{G} = p^{n+1}\mathfrak{g}$, and with corresponding Iwasawa algebras $RG$ and $KG$. Then as discussed in \S \ref{Iwasawa} $KG$ embeds in a completed enveloping algebra $\Ug$.

Now, for any $I \subseteq \Delta$ and any $\mathfrak{h}' \subseteq \mathfrak{h}$ with $\mathfrak{h} = \mathfrak{h}' \oplus \mathfrak{a}$ and $h_{\alpha} \in \mathfrak{h}'$ for all $\alpha \in I$,
there is a corresponding subalgebra $\mathfrak{p}_{I, \mathfrak{h}'}$ defined in \S\ref{gvm}. In the case where $\mathfrak{g}$ is semisimple, this is simply the parabolic subalgebra corresponding to $I$. If $\lambda : \mathfrak{h}_{K} \rightarrow K$ is a weight such that $\lambda(h_{\alpha}) \in \mathbb{N}_{0}$ for all $\alpha \in I$, then we define a corresponding generalised Verma module $M_{I,\mathfrak{h}'}(\lambda)$. In the case where $\mathfrak{g}$ is semisimple, this corresponds to the usual meaning of a generalised Verma module.

\begin{mydef}
An \emph{affinoid generalised Verma module} for $KG$ is a $KG$-module of the form $\widehat{{M}_{I,\mathfrak{h}'}(\lambda)} = \Ug \otimes_{\ug} M_{I,\mathfrak{h}'}(\lambda)$, where $M_{I,\mathfrak{h}'}(\lambda)$ is a generalised Verma module as above, and $\lambda$ satisfies the condition $\lambda(p^{n} \mathfrak{h}_{R}) \subseteq R$.
\end{mydef}

\begin{mydef}
We say that the subset $I \subseteq \Delta$ is a \emph{totally proper} subset if in the Dynkin diagram of $\Phi$ with nodes $\Delta$, no connected component is contained entirely in $I$.
\end{mydef}

\noindent We can now state our main theorem, which is a more general version of theorem \ref{theorem1} that applies to the reductive case.

\begin{mythm}
\label{mainthm}
Any affinoid generalised Verma module $\widehat{{M}_{I,\mathfrak{h}'}(\lambda)}$ such that $I \subset \Delta$ is totally proper is faithful as a $KG$-module.
\end{mythm}

We note that the condition that $I$ is totally proper is in fact necessary. For example, if $I = \Delta$ then $\widehat{{M}_{I,\mathfrak{h}'}(\lambda)}$ is finite-dimensional and so is certainly not faithful over $KG$. Also note that the main case of interest is the case where $\mathfrak{g}_{K}$ is simple, in which our definition of a generalised Verma module is the same as the usual definition. In this case the condition that $I$ is totally proper just means that $I$ is a proper subset of $\Delta$, which is equivalent to the condition that the generalised Verma module is infinite-dimensional over $K$.

Our proof strategy is by induction on the size of the set $I$. Although we are most interested in the case where $\mathfrak{g}_{K}$ is simple, the more general statement is necessary for our induction proof.

\begin{itemize}

\item In \S \ref{reductionsection} we carry out a reduction to the semisimple case. Specifically, given $M_{I,\mathfrak{h}'}(\lambda)$ we associate an affinoid generalised Verma module $\widehat{M_{I}(\mu)}$ for $\mathfrak{g}_{s}$. We then show that for $\widehat{{M}_{I,\mathfrak{h}'}(\lambda)}$ to be faithful over $KG$, it is sufficient that $\widehat{{M}_{I}(\mu)}$ is faithful over $KG_{s}$, where $G_{s}$ is the subgroup of $G$ corresponding to $p^{n+1}\mathfrak{g}_{s}$.

\item This is enough to prove the base case for our induction. The case where $\mathfrak{g}$ is semisimple and $I = \emptyset$ is precisely the case of an affinoid Verma module, which is proved in \cite[Theorem 5.4]{verma}.

\item In \S \ref{dualsection} we restrict to the semisimple case. We use the duality functor on category $\mathcal{O}$ together with Corollary \ref{5.3cor} (which is our version of \cite[Theorem 5.3]{verma}) to show that for an affinoid generalised Verma module to be faithful over $KG$, it is sufficient that it is faithful over $RL$, with $L$ the corresponding Levi subgroup.

\item Looking at the structure of an affinoid generalised Verma module over the Levi subalgebra $\mathfrak{l}$, we note that it contains a family of infinitely many finite-dimensional modules. In \S \ref{annsection} we argue that any element of $\Ul$ annihilating all of these finite-dimensional modules must also annihilate some induced module for $\Ul$.

\item This induced module has the required form for us to apply the inductive hypothesis. This motivates stating the theorem for reductive Lie algebras, because $\mathfrak{l}$ will be reductive but  not semisimple.

\end{itemize}

\section{Background}

\subsection{Preliminaries about Lie algebras}

Whenever $\mathfrak{g}$ is a Lie algebra over a commutative ring $A$ and $B$ is an $A$-algebra, we will write $\mathfrak{g}_{B}$ for the $B$-Lie algebra $\mathfrak{g}_{B} = B \otimes_{A} \mathfrak{g}$.

If $\mathfrak{g}$ is an $F$-Lie algebra, we say an $\Of$-Lie subalgebra $\mathfrak{g}'$ is an $\Of$-lattice of $\mathfrak{g}$ if there is an $F$-basis $x_{1},\dots,x_{r}$ of $\mathfrak{g}$ such that $\mathfrak{g}' = \Of \{x_{1},\dots,x_{r} \}$. 

We now make a statement of the PBW theorem, which is well known in the case of Lie algebras over a field, but also holds for free Lie algebras over commutative rings.

\begin{mythm}
\label{PBWthm}
Suppose $\mathfrak{g}$ is a free, finite rank Lie algebra over a commutative ring $A$, with a basis $x_{1},\dots,x_{m}$. Then $U(\mathfrak{g})$ has a free $A$-basis $\{x_{1}^{s_{1}} \dots x_{m}^{s_{m}} \mid s \in \mathbb{N}_{0}^{m} \}$.
\end{mythm}

\begin{proof}
\cite[2.1.12]{Dixmier}
\end{proof}

\noindent For some results we will want to consider highest-weight modules over Lie algebras that are not semisimple. To this end, we establish a more general setting to define highest-weight modules.

Firstly, suppose $\mathfrak{a}$ is an abelian $K$-Lie algebra and $M$ is a $U(\mathfrak{a})$-module. 

\begin{mydef}
\begin{itemize}
\item Suppose $\lambda : \mathfrak{a}_{K} \rightarrow K$ is a $K$-linear map. We say $v \in M$ is a \emph{weight vector} of weight $\lambda$ if $a \cdot v = \lambda(a) v$ for all $a \in \mathfrak{a}_{K}$.

\item Given $\lambda : \mathfrak{a}_{K} \rightarrow K$ we write $M_{\lambda}$ for the $K$-subspace of $M$ consisting of all $\lambda$-weight vectors.

\item We say $M$ is a \emph{weight module} if $M$ has a $K$-basis consisting of weight vectors.

\end{itemize}
\end{mydef}

\begin{mydef}
\label{triang}
We say an $\Of$-Lie algebra $\mathfrak{g}$ has a \emph{triangular decomposition} if there are free, finite-rank $\Of$-Lie subalgebras $\mathfrak{n}^{+}, \mathfrak{h}, \mathfrak{n}^{-}$ satisfying the following conditions:

\begin{enumerate}
\item $\mathfrak{h}$ is abelian. Given this, we can consider weight vectors and weight modules over $\uh$.

\item $\mathfrak{g} = \mathfrak{n}^{+} \oplus \mathfrak{h} \oplus \mathfrak{n}^{-}$ as an $\Of$-module.

\item Suppose $\mathfrak{n}$ is either $\mathfrak{n}^{+}$ or $\mathfrak{n}^{-}$. Then $[\mathfrak{h}, \mathfrak{n}] \subseteq \mathfrak{n}$. Moreover, the $\Of$-linear maps $[h,-] : \mathfrak{n} \rightarrow \mathfrak{n}$ for $h \in \mathfrak{h}$ are simultaneously diagonalisable. Note that this condition together with Theorem \ref{PBWthm} imply that $\un$ is a weight module with $h \in \mathfrak{h}_{K}$ acting by $[h,-]$.

\item Suppose $\mathfrak{n}$ is either $\mathfrak{n}^{+}$ or $\mathfrak{n}^{-}$. Suppose $x_{1},\dots,x_{r}$ is an $\Of$-basis of weight vectors for $\mathfrak{n}$ with corresponding weights $\alpha_{1},\dots,\alpha_{r} : \mathfrak{h}_{K} \rightarrow K$. Then for any $c_{1},\dots,c_{r} \in \mathbb{N}_{0}$ not all zero we have $\sum c_{i} \alpha_{i} \neq 0$. In particular, for any $\lambda \in \mathfrak{h}_{K}^{*}$, the subspace $\un_{\lambda}$ is finite-dimensional.

\end{enumerate}

\noindent Note for instance that the usual definition of a triangular decomposition for a semisimple Lie algebra meets these conditions.

Now, if $M$ is a $\ug$-module, we say a weight vector $x \in M$ is a \emph{highest-weight vector if} $e \cdot x = 0$ for all $e \in \mathfrak{n}^{+}$. We say $M$ is a \emph{highest-weight module} if $M$ is generated by a highest-weight vector.

\end{mydef}

Now restrict to the case where $\mathfrak{g}_{F}$ is a split semisimple $F$-Lie algebra with root system $\Phi$.

\begin{mydef}
\label{chevalley}
We say that an $\Of$-basis $\{h_{\alpha} \mid \alpha \in \Delta \} \cup \{e_{\alpha},f_{\alpha} \mid \alpha \in \Phi^{+} \}$ for $\mathfrak{g}$
is a \emph{Chevalley basis} if the following relations are satisfied:

\begin{itemize}
\item $[h_{\alpha}, h_{\beta}] = 0$ for all $\alpha, \beta \in \Delta$,

\item $[h_{\alpha}, e_{\beta}] = \langle \beta, \alpha^{\lor} \rangle e_{\beta}$ for all $\alpha \in \Delta, \beta \in \Phi^{+}$,

\item $[h_{\alpha}, f_{\beta}] = -\langle \beta, \alpha^{\lor} \rangle f_{\beta}$ for all $\alpha \in \Delta, \beta \in \Phi^{+}$,

\item $[e_{\alpha},f_{\alpha}] = h_{\alpha}$ for all $\alpha \in \Delta$.

\item Suppose we write $e_{-\alpha} = f_{\alpha}$ whenever $\alpha \in \Phi^{+}$. Then for $\alpha, \beta \in \Phi$, we have $[e_{\alpha},e_{\beta}] = C_{\alpha,\beta}e_{\alpha + \beta}$ for $C_{\alpha,\beta} \in  \Of$, where $C_{\alpha, \beta} \neq 0 $ if and only if $\alpha + \beta \in \Phi$, and where $C_{\alpha,\beta} =  C_{-\beta, -\alpha}$ for any $\alpha, \beta$.

\end{itemize}
\end{mydef}

\noindent This is analagous to the concept of a Chevalley basis for a semisimple Lie algebra over an algebraically closed field of characteristic $0$, which is known to always exist (see for instance  \cite[\S 25.2]{otherHumphreys}). In our setting we will assume that a Chevalley basis exists. For example we could take $\mathfrak{g} = \mathfrak{sl}_{m+1}(\Of)$ for some $m$, which has a Chevalley basis.

We now look at some subalgebras of $\mathfrak{g}$. Given a subset $I \subseteq \Delta$, we can then define a corresponding root subsystem $\Phi_{I} = \mathbb{Z}I \cap \Phi$. Then we have the following subalgebras of $\mathfrak{g}$:

\begin{mydef}
\begin{itemize}
\label{subalgebras}
\item $\mathfrak{h}$ will denote the Cartan subalgebra spanned by $\{h_{\alpha} \mid \alpha \in \Delta \}$.

\item $\mathfrak{p}_{I}$ is the parabolic subalgebra spanned by $\{h_{\alpha} \mid \alpha \in \Delta \} \cup \{e_{\alpha} \mid \alpha \in \Phi^{+} \} \cup \{f_{\alpha} \mid \alpha \in \Phi_{I}^{+} \}$,

\item $\mathfrak{l}_{I}$ is the Levi subalgebra spanned by $\{h_{\alpha} \mid \alpha \in \Delta \} \cup \{e_{\alpha},f_{\alpha} \mid \alpha \in \Phi_{I}^{+} \}$,

\item $\mathfrak{g}_{I}$ is the semisimple subalgebra spanned by $\{h_{\alpha} \mid \alpha \in I \} \cup \{e_{\alpha},f_{\alpha} \mid \alpha \in \Phi_{I}^{+} \}$.

\end{itemize}

\end{mydef}

\noindent Note that $\mathfrak{p}_{I}$ and $\mathfrak{l}_{I}$ have triangular decompositions, so we can define highest-weight modules over them although they are not semisimple.

\subsection{Completions}
\label{completions}
For now, we take $\mathfrak{g}$ to be any $\Of$-Lie algebra which is free of finite rank as an $\Of$-module. 

\begin{mydef}

We fix a descending filtration of $\ug$ as follows: first take $\Gamma_{0}\ug = U(p^{n} \mathfrak{g}_{R})$. Then take $\Gamma_{i}\ug = \pi^{i} \Gamma_{0} \ug$ for each $i \in \mathbb{Z}$.

Define $\Ug$ to be the completion of $\ug$ with respect to this filtration, which is itself a filtered ring with a filtration $\Gamma_{i}\Ug$.

\end{mydef}

\begin{mylem}
\label{PBW}
Suppose we have an $\Of$-basis $x_{1},\dots,x_{r}$ for $\mathfrak{g}$. Write $(p^{n} x)^{s} = (p^{n} x_{1})^{s_{1}} \dots (p^{n} x_{n})^{s_{n}}$ for $s \in \mathbb{N}_{0}^{r}$, and use the notation $|s| = \underset{i}{\sum}s_{i}$. Then as a $K$-vector space, we can identify $\Ug$ with the ring of convergent power series:
 $$\Ug = \{ \underset{s \in \mathbb{N}_{0}^{r}}{\sum} C_{s} (p^{n} x)^{s} \mid C_{s} \in K, C_{s} \rightarrow 0 \text{ as } |s| \rightarrow \infty \}.$$
\end{mylem}
 
\begin{proof}
The PBW theorem (Theorem \ref{PBWthm}) implies that $\Gamma_{0}\ug$ has a free $R$-basis $\{ (p^{n}x)^{s} \mid s \in \mathbb{N}_{0}^{r} \}$. Then $\Gamma_{i} \ug$ has free $R$-basis $\{ \pi^{i}(p^{n}x)^{s} \mid s \in \mathbb{N}_{0}^{r} \}$. The result now follows from the definition of $\Ug$.

\end{proof}

\begin{mydef}
For any finitely generated $\ug$-module $M$, we define the completion of $M$ as $\hat{M} = \Ug \otimes_{\ug} M$. 
\end{mydef}

\begin{myprop}
\label{flatnessprop}
$\Ug$ is flat as a right $\ug$-module. In particular, the completion functor $M \mapsto \hat{M}$ is an exact functor from finitely generated $\ug$-modules to $\Ug$-modules.
\end{myprop}

\begin{proof}
First, note that $\Gamma_{0}\Ug =  \{ \underset{s \in \mathbb{N}_{0}^{r}}{\sum} C_{s} (p^{n} x)^{s} \mid C_{s} \in R, C_{s} \rightarrow 0 \text{ as } |s| \rightarrow \infty \} $ is the $\pi$-adic completion of $\Gamma_{0}\ug$. Also note that $\Gamma_{0} \ug = U(p^{n} \mathfrak{g}_{R})$ is a left Noetherian ring, because its associated graded ring is a polynomial ring over $R$ by the PBW theorem. So by \cite[3.2.3 (iv)]{Dmodules} we can see that $\Gamma_{0} \Ug$ is flat as a right $\Gamma_{0} \ug$-module. 

Now if $M$ is a $\ug$-module, $\Ug \otimes_{\ug} M \cong \Gamma_{0} \Ug \otimes_{\Gamma_{0} \ug} M$ as abelian groups, and so $\Ug$ is flat over $\ug$.
\end{proof}

\begin{myprop}
Suppose $M$ is a finitely-generated $\ug$-module with a good filtration $\Gamma_{i}M$ with respect to the filtration of $\ug$ (i.e. $\Gamma_{i} = \pi^{i} \Gamma_{0}M$ for all $i$ and $\Gamma_{0}M$ is a finitely generated $U(p^{n} \mathfrak{g}_{R})$-module). Then $\hat{M}$ is the completion of $M$ with respect to $\Gamma$. Thus $\hat{M}$ also has a filtration $\Gamma_{i} \hat{M}$.
\end{myprop} 

\begin{proof}
Write $N$ for the completion of $M$ with respect to $\Gamma$. Then $N$ has a filtration $\Gamma$, with $\Gamma_{j}N = \varprojlim_{i \geq j}\Gamma_{i}M = \varprojlim_{i \geq j} \pi^{i} \Gamma_{0}M$. 

So $\Gamma_{0} N$ is the $\pi$-adic completion of $\Gamma_{0}M$, a finitely generated $\Gamma_{0} \ug$-module. Then by \cite[3.2.3 (iii)]{Dmodules}, the natural map $\Gamma_{0} \Ug \otimes_{\Gamma_{0} \ug} \Gamma_{0}M \rightarrow \Gamma_{0} N$ is an isomorphism.

Now note that $\Ug \otimes_{\ug} M = K \otimes_{R} (\Gamma_{0} \Ug \otimes_{\Gamma_{0} \ug} \Gamma_{0}M)$ and $N = K \otimes_{R} \Gamma_{0} N$. Since $K$ is flat over $R$ (being the field of fractions of $R$), we obtain an isomorphism $\Ug \otimes_{\ug} M \rightarrow N$.

\end{proof}

\noindent From now on, we will restrict ourselves to filtered modules in the following category.

\begin{mydef}
The \emph{category} $\mathcal{M}_{n}(\mathfrak{g})$ is the category of all finitely generated filtered $\ug$-modules $(M, \Gamma)$ such that $\Gamma$ is a good filtration, and such that $\Gamma_{0}M$ is free as an $R$-module. The morphisms are filtered module homomorphisms. We will write $M \in \mathcal{M}_{n}(\mathfrak{g})$ if there exists a filtration for $M$ such that $(M,\Gamma) \in \mathcal{M}_{n}(\mathfrak{g})$.

\end{mydef}

\begin{mylem}
\label{separatedlem}
Suppose $(M,\Gamma) \in \mathcal{M}_{n}(\mathfrak{g})$. Then  the filtration $\Gamma$ is separated. This has the following consequences:

\begin{itemize}
\item $M$ embeds as a $\ug$-submodule of $\hat{M}$,

\item the topology of $\hat{M}$ induced by the filtration is Hausdorff.

\end{itemize}
\end{mylem}

\begin{proof}
Suppose $\mathcal{B}$ is a free $R$-basis for $\Gamma_{0}M$. We can describe the filtration of $M$ as follows:  $\Gamma_{i}M = \bigoplus_{v \in \mathcal{B}} \pi^{i}R v$, so $\underset{i \in \mathbb{Z}}{\bigcap} \Gamma_{i}M = 0$.
\end{proof}

Let $M \in \mathcal{M}_{n}(\mathfrak{g})$. Given a free $R$-basis $\mathcal{B}$ for $\Gamma_{0}M$, every element of $\hat{M}$ can be uniquely described as a sum $\underset{v \in \mathcal{B}}{\sum} C_{v} v$ with $v \in M$, which is convergent in the sense that for any $i$, $C_{v} \in \pi^{i}R$ for all but finitely many $v$. Moreover, every such convergent sum describes an element of $\hat{M}$. We can then describe $M \subseteq \hat{M}$ as the subset of all the finite sums.

\begin{mylem}
\begin{enumerate}

\item Suppose $M$ is a $\ug$-module which is finite-dimensional over $K$. Then the embedding $M \subseteq \hat{M}$ is a bijection, and so $M$ can be considered as a $\Ug$-module.

\item Suppose $M \in \mathcal{M}_{n}(\mathfrak{g})$ and $N \leq M$ is a finite-dimensional $\ug$-submodule. Then $\Ug \cdot N \subseteq N$ in $M$.
\end{enumerate}
\end{mylem}

\begin{proof}

\begin{enumerate}
\item Suppose $M$ has a $K$-basis $x_{1},\dots,x_{r}$. Then define a good filtration $\Gamma$ on $M$ by $\Gamma_{i}M = \pi^{i} R \{x_{1},\dots,x_{r} \}$. So $(M,\Gamma) \in \mathcal{M}_{n}(\mathfrak{g})$. Then our description of $\hat{M}$ and $M$ above means that both $\hat{M}$ and $M$ consist of all sums $\sum a_{i} x_{i}$ with $a_{i} \in K$, that is, $\hat{M} = M$.

\item By exactness of the completion functor (Proposition \ref{flatnessprop}), we can see that the embedding $N \subseteq M$ of $\ug$-modules extends to an embedding $\hat{N} \subseteq \hat{M}$ of $\Ug$-modules. But since $\hat{N} = N$ as above, this means that $N \subseteq \hat{M}$ is a $\Ug$-submodule.

\end{enumerate}

\end{proof}

\noindent The following result will be useful when considering generalised Verma modules, which we will see are locally finite over the corresponding parabolic subalgebra.

\begin{myprop}
\label{locallyfiniteprop}
Suppose $(M, \Gamma) \in \mathcal{M}_{n}(\mathfrak{g})$ and $\mathfrak{p}$ is an $\Of$-Lie subalgebra of $\mathfrak{g}$ such that $M$ is a locally finite $\up$-module. Then $\Up \cdot M \subseteq M$ inside $\hat{M}$.

\end{myprop}

\begin{proof}
Let $w \in M$, and we claim that $\Up \cdot w \subseteq M$. Let $\mathcal{B}$ be a free $R$-basis for $\Gamma_{0}M$.

Firstly, consider $\up \cdot w$. Each element of $\up \cdot w$ can be described as a finite linear combination of elements of $\mathcal{B}$, and moreover $\up \cdot w$ is finite-dimensional by local finiteness. So we can choose $v_{1},\dots,v_{r} \in \mathcal{B}$ such that $\up \cdot w \subseteq K \{v_{1},\dots,v_{r} \} \subseteq M$. Then any element of $\Up \cdot w$ is the limit of a sequence of elements of $\up \cdot w \subseteq K \{v_{1},\dots,v_{r} \}$, and hence also lies in $K \{v_{1},\dots,v_{r} \} \subseteq M$.
\end{proof}

\noindent When proving that an element of $\Ug$ annihilates $\hat{M}$, it can be useful to use the following result, which comes from $M$ being a  dense subspace in $\hat{M}$ in the topology of the filtration $\Gamma$.

\begin{mylem}
\label{densitylem}
Let $(M,\Gamma) \in \mathcal{M}_{n}(\mathfrak{g})$. Suppose $x \in \Ug$ such that $x \cdot M = 0$ in $\hat{M}$. Then $x \cdot \hat{M} = 0$.

\end{mylem}

\begin{proof}
Let $v \in \hat{M}$. Say $v = \underset{i \geq j}{\sum} v_{i}$ where $j \in \mathbb{Z}$ and $v_{i} \in \Gamma_{i}M$ for all $i$. Also say $x \in \Gamma_{m}\Ug$ (using the filtration on $\Ug$ coming from the filtration on $\ug$).

Then $x \cdot v = x \cdot v - x \cdot \underset{j \leq i \leq k}{\sum} v_{i} = x \cdot \underset{i > k}{\sum}v_{i} \in \Gamma_{m+k+1} \hat{M}$ for each $k \geq j$. So $x \cdot v = 0$ since $\Gamma$ is separated by Lemma \ref{separatedlem}. 

\end{proof}

\noindent \textbf{For the remainder of the section we want to specialise further. So now assume that $\mathfrak{g}$ is a free, finite-rank $\Of$-Lie algebra with a triangular decomposition in the sense of Definition \ref{triang}}. 

We can now define a category of modules in our setting analogous to the usual concept of category $\mathcal{O}$  for split semisimple Lie algebras.

\begin{mydef}
\label{categoryO}
The \emph{category} $\mathcal{O}(\mathfrak{g})$ is the full subcategory of all $\ug$-modules $M$ satisfying:

\begin{itemize}
\item $M$ is a weight module,

\item $M$ is finitely generated over $\ug$,

\item $M$ is locally finite over $U(\mathfrak{n}^{+}_{K})$.

\end{itemize}

\noindent The \emph{category} $\mathcal{O}_{n}(\mathfrak{g})$ is the full subcategory of $\mathcal{O}(\mathfrak{g})$ consisting of modules $M$  such that all weights $\lambda$ of $M$ satisfy $\lambda(p^{n} \mathfrak{h}_{R}) \subseteq R$. 

\end{mydef}

\noindent In particular, $\mathcal{O}_{n}(\mathfrak{g})$ contains any highest-weight module of weight $\lambda$ such that $\lambda(p^{n} \mathfrak{h}_{R}) \subseteq R$. The following results are standard facts about category $\mathcal{O}$.

\begin{mylem}
\label{catObasics}
Suppose $M \in \mathcal{O}(\mathfrak{g})$. Then:

\begin{itemize}
\item $M$ is generated by a finite set of weight vectors,

\item the weight spaces of $M$ are finite-dimensional.

\end{itemize}
\end{mylem}

\begin{proof}
Since $M \in \mathcal{O}(\mathfrak{g})$, $M$ is finitely generated. Moreover, $M$ is a weight module. So let $T$ be a finite generating set for $M$, then let $S$ be the set of all weight components of elements of $T$. Thus $S$ is a finite generating set for $M$ consisting of weight vectors.

Now, since $M$ is locally finite over $U(\mathfrak{n}^{+}_{K})$, we have that $U(\mathfrak{n}^{+}_{K}) S$ is finite-dimensional and spanned by weight vectors. Say $S'$ is a $K$-basis for $U(\mathfrak{n}^{+}_{K}) S$ consisting of weight vectors. Then $S'$ generates $M$. 

Moreover, by Theorem \ref{PBWthm}, $M = \ug S = U(\mathfrak{n}^{-}_{K}) U(\mathfrak{n}^{+}_{K}) \uh S$. Since the elements of $S$ are weight vectors, $M = U(\mathfrak{n}^{-}_{K}) U(\mathfrak{n}^{+}_{K})S = U(\mathfrak{n}^{-}_{K})S'$. By Definition \ref{triang}, the weight spaces of $U(\mathfrak{n}^{-}_{K})$ are finite-dimensional. Since $S'$ is finite, the weight spaces of $M$ are finite-dimensional: if we write $\lambda_{v}$ for the weight of $v \in S'$, then $M_{\lambda} = \underset{v \in S'}{\sum} U(\mathfrak{n}^{-}_{K})_{\lambda - \lambda_{v}} v$ for any weight $\lambda$.
\end{proof}

\noindent We now want to show that any $M \in \mathcal{O}_{n}(\mathfrak{g})$ has a filtration $\Gamma$ such that $(M,\Gamma) \in \mathcal{M}_{n}(\mathfrak{g})$, so that our results for $\mathcal{M}_{n}(\mathfrak{g})$ remain applicable.

\begin{myprop}
\label{catOstructure}
Suppose $M$ is a $\ug$-module in category $\mathcal{O}_{n}(\mathfrak{g})$. Let $S$ be a finite set of weight vectors generating $M$. Then $U(p^{n} \mathfrak{g}_{R})S$ is free as an $R$-module, with a free $R$-basis consisting of weight vectors.
\end{myprop}

\begin{proof}
We have $U(p^{n} \mathfrak{g}_{R})S = \bigoplus_{\lambda \in \mathfrak{h}_{K}^{*}} U(p^{n} \mathfrak{g}_{R})S \cap M_{\lambda}$ since the elements of $S$ are weight vectors, and hence $U(p^{n} \mathfrak{g}_{R})S$ is spanned by weight vectors. So it is sufficient to show the weight spaces of $U(p^{n} \mathfrak{g}_{R})S$ are free over $R$. 

Now, writing $\mathfrak{n}^{-}$ for the $\mathcal{O}_{F}$-subalgebra of $\mathfrak{g}$ spanned by the $f_{\alpha}$, we have $U(p^{n} \mathfrak{g}_{R})S = U(p^{n} \mathfrak{n}^{-}_{R}) U(p^{n} \mathfrak{n}^{+}_{R}) U(p^{n} \mathfrak{h}_{R}) S$ by the PBW theorem. 

The assumption that all weights $\lambda$ of $M$ satisfy $\lambda(p^{n} \mathfrak{h}_{R}) \subseteq R$ ensures $U(p^{n} \mathfrak{h}_{R}) S = R S$, so $U(p^{n} \mathfrak{g}_{R})S = U(p^{n} \mathfrak{n}^{-}_{R}) U(p^{n} \mathfrak{n}^{+}_{R}) S$.

By assumption, both $\mathfrak{n}^{+}$ and $\mathfrak{n}^{-}$ have free $\Of$-bases of weight vectors, say $e_{1},\dots,e_{r}$ and  $f_{1},\dots,f_{r'}$ respectively. So by Theorem \ref{PBWthm}, $U(p^{n} \mathfrak{n}^{+}_{R})S$ is spanned as an $R$-module by the set of weight vectors $S' = \{ (p^{n} e_{1})^{i_{1}} \dots (p^{n} e_{r})^{i_{r}} s \mid   i_{j} \in \mathbb{N}_{0}, s \in S \}$. But note that since $U(\mathfrak{n}_{K}^{+})$ has finite-dimensional weight spaces (by our definition of a triangular decomposition), only finitely many monomials $(p^{n} e_{1})^{i_{1}} \dots (p^{n} e_{r})^{i_{r}}$ can have the same weight. Therefore only finitely many elements of $S'$ can have the same weight. 

But since $M$ is locally finite over $U(\mathfrak{n}^{+}_{K})$, elements of $S'$ can only have finitely many different weights, and so $S'$ has only finitely many non-zero elements. So $U(p^{n} \mathfrak{g}_{R})S = U(p^{n} \mathfrak{n}^{-}_{R})S'$ with $S'$ a finite set of weight vectors.

Now, fix a weight $\lambda$ of $M$. Then $U(p^{n} \mathfrak{g}_{R})S \cap M_{\lambda}$ is spanned by the set of all $(p^{n} f_{1})^{i_{1}} \dots (p^{n} f_{r'})^{i_{r}} s'$ with $ i_{j} \in \mathbb{N}_{0}, s' \in S'$ of weight $\lambda$. There are only finitely many such elements since the weight spaces of $U(\mathfrak{n}^{-}_{K})$ are finite-dimensional and since $S'$ is finite, so $U(p^{n} \mathfrak{g}_{R})S \cap M_{\lambda}$ is finitely generated over $R$.

Now, $U(p^{n} \mathfrak{g}_{R})S \cap M_{\lambda}$ is a finitely generated torsion-free module over a principal ideal domain $R$, and so is free of finite rank over $R$.

\end{proof}

\begin{mycor}
\label{catOincatM}
Suppose $M \in \mathcal{O}_{n}(\mathfrak{g})$. Then there is a good filtration $\Gamma$ for $M$ such that $(M,\Gamma) \in \mathcal{M}_{n}(\mathfrak{g})$, and also $\Gamma_{0}M$ has a free $R$-basis of weight vectors.

\end{mycor}

\begin{proof}
Let $S$ be a finite generating set of $M$ consisting of weight vectors, which exists by Lemma \ref{catObasics}.

Define $\Gamma_{i}M = \Gamma_{i}\ug S$ for all $i \in \mathbb{Z}$, so $\Gamma$ is a good filtration for $M$. Moreover, by Proposition \ref{catOstructure}, $\Gamma_{0}M = U(p^{n} \mathfrak{g}_{R})S$ has a free $R$-basis of weight vectors.

\end{proof}

\noindent If a basis $x_{1},\dots,x_{r}$ for $\mathfrak{g}$ consists of weight vectors, (e.g. a Chevalley basis when $\mathfrak{g}$ is semisimple), then we can describe weight components of $y \in \Ug$ extending the concept of weight components in $\ug$.

\begin{mydef}
If $y = \sum C_{s} (\pi x)^{s} \in \Ug$ and $\lambda$ is a weight, then the $\lambda$-\emph{weight component} is the convergent sum of all $C_{s} (\pi x)^{s}$ which have weight $\lambda$. We also say that $y$ has weight $\lambda$ if all weight components of $y$ with weight not equal to $\lambda$ are zero. We then say $y$ is a \emph{weight vector} if $y$ has weight $\lambda$ for some $\lambda$. 

\end{mydef}

\begin{mylem}
\label{weightvectorlem}
Suppose $M \in \mathcal{O}_{n}(\mathfrak{g})$. Suppose $x \in \Ug$ is a weight vector. Then $x \cdot M \subseteq M$ inside $\hat{M}$.
\end{mylem}

\begin{proof}

Let $S$ be a finite generating set of weight vectors of $M$, which exists by Lemma \ref{catObasics}. Let $\Gamma$ be the corresponding good filtration, so $\Gamma_{i}M = \Gamma_{i}\ug S$. Let $\mathcal{B}$ be a free $R$-basis for $\Gamma_{0}M$ consisting of weight vectors (which exists by Proposition \ref{catOstructure}).

Suppose $w \in M$. Then $w$ has finitely many non-zero weight components, since $M$ is a weight module. Say $W$ is the set of all weights of components of $w$. Say $x$ has weight $\lambda$, and let $\mathcal{B}'$ be the set of all elements of $\mathcal{B}$ with weights lying in $\lambda + W$. Note that since the weight spaces of $M$ are finite-dimensional by Lemma \ref{catObasics}, only finitely many elements of $\mathcal{B}$ can have the same weight and therefore $\mathcal{B}'$ is finite.

Now we can write $x = \lim_{i} x_{i}$ where $x_{i} \in \ug$ have weight $\lambda$. Then all weight components of $x_{i} \cdot w$ have weights in $\lambda + W$, and therefore $x_{i} \cdot w \in K \mathcal{B}'$. Therefore $x \cdot w = \lim_{i} x_{i} \cdot w \in K \mathcal{B}' \subseteq M$.
\end{proof}

\begin{mylem}
\label{weightcompslem}
Let $M \in \mathcal{O}_{n}(\mathfrak{g})$. Suppose $x \in \Ann_{\Ug} \hat{M}$. Then all weight components of $x$ lie in $\Ann_{\Ug} \hat{M}$.
\end{mylem}

\begin{proof}
Say $x = \underset{\lambda}{\sum} x_{\lambda}$ as a convergent sum, where $x_{\lambda} \in \Ug$ has weight $\lambda$. We claim $x_{\lambda} \cdot M = 0$ for each $\lambda$, which by Lemma \ref{densitylem} will imply $x_{\lambda} \in \Ann_{\Ug} \hat{M}$.

Let $v \in M$ be a weight vector of weight $\mu$. Then $0 = x \cdot v = \underset{\lambda}{\sum} x_{\lambda} \cdot v$. Then the $(\lambda + \mu)$-weight component of $x \cdot v$ is $x_{\lambda} \cdot v$ for each $\lambda$, so $x_{\lambda} \cdot v = 0$ for all $\lambda$. Thus indeed each $x_{\lambda}$ annihilates $M \subseteq \hat{M}$.
\end{proof}

\noindent We will now claim that any $M \in \mathcal{O}_{n}(\mathfrak{g})$ is diagonalisable over $\Uh$. Suppose $\lambda : p^{n} \mathfrak{h}_{R} \rightarrow R$, then $\lambda$ defines a filtered homomorphism $\lambda : \uh \rightarrow K$ (with the filtration $\pi^{i} R$ on $K$). So this extends to a filtered homomorphism $\lambda : \Uh \rightarrow K$.

\begin{mylem}
\label{diagonalmult}
Suppose $M \in \mathcal{O}_{n}(\mathfrak{g})$ and $v \in M$ has weight $\lambda \in \mathfrak{h}_{K}^{*}$ such that $\lambda( p^{n} \mathfrak{h}_{R}) \subseteq R$. Then for any $x \in \Uh$ we have $x \cdot v = \lambda(x) v$ in $\hat{M}$.
\end{mylem} 

\begin{proof}

Say $x = \underset{i \geq j}{\sum} x_{i}$ where $j \in \mathbb{Z}$ and $x_{i} \in \Gamma_{i} \uh$ for each $i$. Say $v \in \Gamma_{m}M$, where $m \in \mathbb{Z}$. Then for each $k \geq j$, we have: 
\begin{equation*}
\begin{split}
x \cdot v - \lambda(x) v & = \underset{j \leq i \leq k}{\sum} x_{i} \cdot v + \underset{i > k}{\sum} x_{i} \cdot v - \lambda(\underset{j \leq i \leq k}{\sum} x_{i}) v - \lambda(\underset{i>k}{\sum} x_{i})v \\
& = \underset{i>k}{\sum} x_{i} \cdot v - \lambda(\underset{i>k}{\sum} x_{i}) v \in \Gamma_{m+k+1}\hat{M}. 
\end{split}
\end{equation*}
\noindent So $x \cdot v = \lambda(x)v.$

\end{proof}

\subsection{Induced modules}
\label{gvm}
In this section, we look at modules induced from finite-dimensional modules. Firstly, we want to confirm that the modules relevant to us are in category $\mathcal{M}_{n}(\mathfrak{g})$ given an appropriate filtration. This allows us to apply the results of \S \ref{completions}.

\begin{myprop}
\label{locallyfinite}
Suppose $\mathfrak{g} = \mathfrak{p} \oplus \mathfrak{n}$ as $\Of$-modules, where both $\mathfrak{p}$ and $\mathfrak{n}$ are free, finite rank $\Of$-Lie algebras. Suppose $(V, \Gamma) \in \mathcal{M}_{n}(\mathfrak{p})$ such that $\Gamma_{0}V$ has finite rank over $R$. Let $M = \ug \otimes_{\up} V$. Then:

\begin{enumerate}
\item $\Gamma$ induces a filtration on $M$ such that $(M, \Gamma) \in \mathcal{M}_{n}(\mathfrak{g})$,

\item $M$ is locally finite as a $\up$-module, which in particular means it can be considered as a (locally finite) $\Up$-module.
\end{enumerate}

\end{myprop}

\begin{proof}
Define $\Gamma_{i}M = U(p^{n} \mathfrak{g}_{R}) \Gamma_{i}V $ . Then $\underset{i}{\bigcup} \Gamma_{i}M = \ug V = M$. Moreover, $\Gamma_{i}M = \pi^{i} U(p^{n} \mathfrak{g}_{R}) \Gamma_{0}V = \pi^{i} \Gamma_{0}M$, and $\Gamma_{0}M$ is finitely generated over $U(p^{n} \mathfrak{g}_{R})$ (e.g. by any finite generating set for $\Gamma_{0}V$ as a $U(p^{n} \mathfrak{p}_{R})$-module). So $\Gamma$ is a good filtration for $M$.

Given a finite free $R$-basis $B$ for $\Gamma_{0}V$ and a free $\Of$-basis $y_{1},\dots,y_{r}$ for $\mathfrak{n}$, the PBW theorem implies that $\Gamma_{0}M$ has free $R$-basis $\{(p^{n}y_{1})^{s_{1}} \dots (p^{n} y_{r})^{s_{r}} v \mid x_{i} \in \mathbb{N}_{0}, v \in B \}$. So $(M, \Gamma) \in \mathcal{M}_{n}(\mathfrak{g})$. This proves the first part.

In order to prove the second part, we claim that the $K$-vector spaces $\Omega_{i}M = K \{ x_{1} \dots x_{j} v \mid 0 \leq j \leq i, x_{1},\dots,x_{j} \in \mathfrak{g}, v \in V \}$ are in fact $\up$-modules. The proposition then follows since each $\Omega_{i}M$ is finite-dimensional, and $\underset{i}{\bigcup} \Omega_{i}M = M$.

We proceed by induction. Firstly, $\Omega_{0}M = V$ is a $\up$-module. Now assume $\Omega_{i}M$ is a $\up$-module, and show $\Omega_{i+1}M$ is a $\up$-module.

Let $y \in \up$, $x_{1},\dots,x_{i+1} \in \mathfrak{g}$ and $v \in V$. Then $y \cdot x_{1} \dots x_{i+1} v = [y,x_{1}] x_{2} \dots x_{i+1} v + x_{1} y x_{2} \dots x_{i+1} v$. 

By definition, $[y,x_{1}] x_{2} \dots x_{i+1} v \in \Omega_{i+1}M$. By our induction assumption, $y x_{2} \dots x_{i+1} v \in \Omega_{i} M$ and hence $x_{1} y x_{2} \dots x_{i+1} \in \Omega_{i+1}M$. So $y \cdot x_{1} \dots x_{i+1} v \in \Omega_{i+1}M$.

We now conclude that every $\Omega_{i}M$ is a $\up$-module, and therefore $M$ is locally finite over $\up$. $M$ now can be considered as a $\Up$-module by Proposition \ref{locallyfiniteprop}.

\end{proof}

\noindent Now note that in the case where $\mathfrak{p}$ has a triangular decomposition and $V \in \mathcal{O}_{n}(\mathfrak{p})$, we know there is a filtration $\Gamma$ for $V$ such that $(V,\Gamma) \in \mathcal{M}_{n}(\mathfrak{p})$. So the above result implies that $M = \ug \otimes_{\up}V $ has a filtration $\Gamma$ such that $(M,\Gamma) \in \mathcal{M}_{n}(\mathfrak{g})$. 

\medskip

\noindent We now look at the specific induced modules that we want to study.

\noindent \textbf{For the rest of the section, let $\mathfrak{g} = \mathfrak{g}_{s} \oplus \mathfrak{a}$ be an $\Of$-lattice of a split reductive $F$-Lie algebra, with root system $\Phi$, a fixed set of simple roots $\Delta$, $\mathfrak{a}$ the centre of $\mathfrak{g}$, and a Chevalley basis  $\{e_{\alpha},f_{\alpha} \mid \alpha \in \Phi^{+} \} \cup \{ h_{\alpha} \mid \alpha \in \Delta \}$ for $\mathfrak{g}_{s}$. Say $\mathfrak{h}_{s}$ is a Cartan subalgebra of $\mathfrak{g}$, and $\mathfrak{h} = \mathfrak{a} \oplus \mathfrak{h}_{s}$.}

Suppose $I \subseteq \Delta$ and $\mathfrak{h}' \subseteq \mathfrak{h}$ is a subalgebra such that $\mathfrak{h} = \mathfrak{h}' \oplus \mathfrak{a}$ and $h_{\alpha} \in \mathfrak{h}'$ for all $\alpha \in I$. Suppose $\lambda: \mathfrak{h}_{K} \to K$ is a weight such that $\lambda(h_{\alpha}) \in \mathbb{N}_{0}$ for all $\alpha \in I$. 

We then have a corresponding subalgebra $\mathfrak{p}_{I,\mathfrak{h}'} = \Of \{e_{\alpha} \mid \alpha \in \Phi^{+} \} \oplus \Of \{f_{\alpha} \mid \alpha \in \Phi_{I}^{+} \} \oplus \mathfrak{h}'$.

Note that the parabolic subalgebra $\mathfrak{p}_{I}$ of $\mathfrak{g}_{s}$ is isomorphic to $\mathfrak{p}_{I,\mathfrak{h}'}$ as a Lie algebra. Specifically, let $f : \mathfrak{g} \rightarrow \mathfrak{g}_{s}$ be the projection map with kernel $\mathfrak{a}$, which is a Lie algebra homomorphism. Then $f$ restricts to an isomorphism $f : \mathfrak{p}_{I,\mathfrak{h}'} \rightarrow \mathfrak{p}_{I}$.

Let $\mathfrak{g}_{I}$ be the semisimple subalgebra of $\mathfrak{g}_{s}$ corresponding to $I$. That is, $\mathfrak{g}_{I}$ is spanned over $\Of$ by  $\{e_{\alpha},f_{\alpha} \mid \alpha \in \Phi_{I}^{+} \} \cup \{ h_{\alpha} \mid \alpha \in I \}$. Then let $V$ be the unique irreducible highest-weight $U(\mathfrak{g}_{I,K})$-module of weight $\lambda |_{\mathfrak{h}_{I,K}}$. The existence and uniqueness of $V$ is a standard fact, see for instance \cite[7.1.11, 7.1.13]{Dixmier}. Then $V$ is finite-dimensional because $\lambda(h_{\alpha}) \in \mathbb{N}_{0}$ for all $\alpha \in I$, see for instance \cite[7.2.6]{Dixmier}. Moreover, $V$ can be made into a $U(\mathfrak{p}_{I,\mathfrak{h'},K})$-module as follows. 

\medskip

\noindent Firstly, let $\mathfrak{h}''_{K} = \underset{\alpha \in I}{\bigcap}\ker \alpha \leq \mathfrak{h}'_{K}$, so  $\mathfrak{h}''_{K}$ is central in $U(\mathfrak{g}_{I,K} \oplus \mathfrak{h}''_{K})$. We can also see that $\mathfrak{h}'_{K} = \mathfrak{h}''_{K} \oplus \mathfrak{h}_{I,K}$. This holds because if $h \in \mathfrak{h}''_{K} \cap \mathfrak{h}_{I,K}$ then $0 = \alpha(h) = \langle h, h_{\alpha}^{\lor} \rangle$ for all $\alpha \in I$, which means $(h,h_{\alpha}) = 0$ for all $\alpha \in I$ and hence $h = 0$ since the $h_{\alpha}$ with $\alpha \in I$ span $\mathfrak{h}_{I,K}$. So $\mathfrak{h}''_{K} \cap \mathfrak{h}_{I,K} = 0$. Moreover by standard linear algebra $\dim_{K} \mathfrak{h}''_{K} = \dim_{K} \underset{\alpha \in I}{\bigcap} \alpha \geq \mathfrak{h}'_{K} - |I|$, which means indeed $\mathfrak{h}'_{K} = \mathfrak{h}''_{K} \oplus \mathfrak{h}_{I,K}$.

By Theorem \ref{PBWthm}, we can see that $U(\mathfrak{g}_{I,K} \oplus \mathfrak{h}''_{K}) \cong U(\mathfrak{g}_{I,K}) \otimes_{K} U(\mathfrak{h}''_{K})$ as a $K$-module, and moreover since $\mathfrak{h}''_{K}$ is central, this is also an isomorphism of $K$-algebras. We can then define a $K$-algebra homomorphism $U(\mathfrak{g}_{I,K}+\mathfrak{h}''_{K}) \rightarrow U(\mathfrak{g}_{I,K})$ by mapping each $h \in \mathfrak{h}''_{K}$ to $\lambda(h)$. Then inflate $V$ along this homomorphism to make $V$ into a $U(\mathfrak{g}_{I,K}+\mathfrak{h}''_{K})$-module, which is a highest-weight module of weight $\lambda$.
 
Then note that the nilpotent subalgebra $\mathfrak{n}_{I} = \Of \{e_{\alpha} \mid \alpha \in \Phi^{+} \setminus \Phi_{I}^{+} \}$ satisfies $\mathfrak{p}_{I,\mathfrak{h}',K} = \mathfrak{g}_{I,K} \oplus \mathfrak{n}_{I,K} \oplus \mathfrak{h}''_{K} $. This is because by definition $\mathfrak{p}_{I,\mathfrak{h}',K} = \Of \{e_{\alpha} \mid \alpha \in \Phi^{+} \} \oplus \mathfrak{h}'_{K} \oplus \{f_{\alpha} \mid \alpha \in \Phi_{I}^{+} \} =  (\Of \{e_{\alpha} \mid \alpha \in \Phi_{I}^{+} \} \oplus \mathfrak{h}_{I,K} \oplus \{f_{\alpha} \mid \alpha \in \Phi_{I}^{+} \} ) \oplus \Of \{ e_{\alpha} \mid \alpha \in \Phi^{+} \setminus \Phi_{I}^{+} \} \oplus \mathfrak{h}''_{K} = \mathfrak{g}_{I,K} \oplus\mathfrak{n}_{I,K}  \oplus \mathfrak{h}''_{K}$.

 Also $\mathfrak{n}_{I}$ is an ideal of $\mathfrak{p}_{I,\mathfrak{h}',K}$. This is because if $\alpha \in \Phi^{+} \setminus \Phi_{I}^{+}$ and $\beta \in \Phi^{+}$, then $\alpha + \beta \notin \Phi_{I}^{+}$ since $\alpha$ and hence $\alpha + \beta$ has a non-zero coefficient of some simple root not in $I$. Thus $[e_{\alpha},e_{\beta}]$ (which is either $0$ or a scalar multiple of $e_{\alpha + \beta}$) lies in $\mathfrak{n}_{I}$. Similarly if $\alpha \in \Phi^{+} \setminus \Phi_{I}^{+}$ and $\beta \in \Phi_{I}^{+}$ then $\alpha - \beta$ has a positive coefficient of some element of $\Delta \setminus I$ and so $[e_{\alpha},f_{\beta}] \in \mathfrak{n}_{I}$.

Inflate $V$ along the projection $\mathfrak{p}_{I,\mathfrak{h}',K} \rightarrow \mathfrak{g}_{I,K}+\mathfrak{h}''_{K}$ with kernel $\mathfrak{n}_{I,K}$, and we make $V$ into an irreducible highest-weight $U(\mathfrak{p}_{I,\mathfrak{h}',K})$-module of weight $\lambda$ such that all $e_{\alpha}$ with $\alpha \in \Phi^{+} \setminus \Phi_{I}^{+}$ act as $0$.
 
\begin{mydef}
The \emph{generalised Verma module} corresponding to $I$ and $\lambda$ is $M_{I,\mathfrak{h}'}(\lambda) = U(\mathfrak{g}) \otimes_{U(\mathfrak{p}_{I,\mathfrak{h}'})} V$.
\end{mydef}

\noindent Suppose $\mathcal{B}$ is a $K$-basis for $V$. Let $x_{1},\dots,x_{r}$ be the elements of $\{f_{\alpha} \mid \alpha \in \Phi \setminus \Phi_{I} \}$. Let $x_{r+1},\dots,x_{s}$ be an $\Of$-basis for $\mathfrak{a}$. Then by the PBW theorem, $M_{I,\mathfrak{h}'}(\lambda)$ has basis $\{x^{t}v\mid t \in \mathbb{N}_{0}^{s}, v \in \mathcal{B} \}$. Here, $x^{t}$ denotes $x_{1}^{t_{1}} \dots x_{s}^{t_{s}}$.

In the case where $\mathfrak{g}_{F}$ is split semisimple (meaning that $\mathfrak{h}'$ must equal $\mathfrak{h}$) we simply write $M_{I}(\lambda)$ for $M_{I,\mathfrak{h}'}(\lambda)$, which agrees with the notation of $\cite{Humphreys}$.

\subsection{$F$-uniform groups}
\label{Funiform}
Suppose $G$ is a uniform pro-p group in the sense of \cite[Definition 4.1]{prop}. Then $G$ has an associated $\mathbb{Z}_{p}$-Lie algebra $L_{G}$. The elements of this Lie algebra are the elements of $G$. The structure is described in full detail in \cite[\S4]{prop}, but we can describe the operations as follows:

\begin{itemize}
\item $a \cdot g = g^{a}$ for $a \in \mathbb{Z}_{p}$, $g \in G$,

\item $g + h = \lim_{i \to \infty} (g^{p^{i}} h^{p^{i}})^{p^{-i}}$ for $g,h \in G$,

\item $[g,h] = \lim_{i \to \infty} (g^{-p^{i}} h^{-p^{i}} g^{p^{i}} h^{p^{i}})^{p^{-2i}}$ for $g,h \in G$.

\end{itemize}

\noindent Any uniform pro-p group can be described using a minimal topological generating set as follows.

\begin{myprop}
Let $G$ be a uniform pro-p group. Then if $\{x_{1},\dots,x_{r} \}$ is a topological generating set for $G$ of minimal size, the map $(\lambda_{1},\dots,\lambda_{r}) \mapsto x_{1}^{\lambda_{1}} \dots x_{r}^{\lambda_{r}}$ is a homeomorphism from $\mathbb{Z}_{p}^{r}$ to $G$.
\end{myprop}

\begin{proof}
\cite[Theorem 4.9]{prop}
\end{proof}

\noindent We can then relate minimal topological generating sets of $G$ and $\mathbb{Z}_{p}$-bases of $L_{G}$.

\begin{myprop}
\label{basisprop}
A subset $\{x_{1},\dots,x_{r} \}$ of $G$ is a minimal topological generating set of $G$ if and only if it is a $\mathbb{Z}_{p}$-basis for $L_{G}$.
\end{myprop}

\begin{proof}
\cite[Theorem 4.17]{prop} gives that any minimal topological generating set for $G$ is a $\mathbb{Z}_{p}$-basis for $L_{G}$. Adapting the proof of this theorem, we can show the converse also holds.

From \cite[Corollary 4.15]{prop}, the multiplicative subgroup $G_{2} = G^{p}$ of $G$ is also an additive subgroup of $G$, and the additive cosets of $G_{2}$ are the same as the multiplicative cosets. By \cite[Proposition 4.16]{prop}, $(G,+)$ is a uniform pro-p group with the same dimension as $G$. By \cite[Lemma 3.4]{prop}, $G_{2}$ is the Frattini subgroup of both $G$ and $(G,+)$.

Suppose $X$ is a $\mathbb{Z}_{p}$-basis for $L_{G}$. Then $X$ topologically generates $(G,+)$. By \cite[Proposition 1.9]{prop}, the image of $X$ in $G/G_{2}$ generates $G / G_{2}$ (which is the same group whether we use the multiplication or the addition operation on $G$). Therefore $X$ topologically generates $G$, and is minimal since the dimensions of $G$, $(G,+)$ as uniform groups are the same.

\end{proof}

\noindent Now make the further assumption that $G$ is an $F$-analytic group. Following \cite[Remark 2.2.5]{Orlik}, we can identify $L_{G}$ with a $\mathbb{Z}_{p}$-submodule of the $F$-Lie algebra $\text{Lie}(G)$, via the exponential map.

\begin{mydef}
A uniform pro-p group $G$ is an \emph{$F$-uniform group} if $G$ is an $F$-analytic group, and moreover $L_{G}$ is an $\Of$-submodule of $\text{Lie}(G)$.

\end{mydef}

\begin{mydef}
We say an $\Of$-Lie algebra $\mathfrak{g}$ is powerful if $[\mathfrak{g},\mathfrak{g}]\subseteq p \mathfrak{g}$.
\end{mydef}

\begin{mythm}
\label{categoryisom}
The assignment $G \mapsto L_{G}$ is an isomorphism of categories from the category of $F$-uniform groups to the category of free, finite rank, torsion-free, powerful $\Of$-Lie algebras. The inverse is given by $L \mapsto (L,*)$ where $x*y = \log(\exp(x)\exp(y))$.
\end{mythm}

\begin{proof}
Follows from \cite[9.10]{prop}.
\end{proof}

\noindent We will use this result repeatedly to show the existence of groups corresponding to certain $\Of$-Lie algebras.

\begin{myex}
Given any $m \in \mathbb{N}$, the group corresponding to $p^{n+1} \mathfrak{sl}_{m+1}(\Of)$ is $\ker( SL_{m+1}(\Of) \rightarrow SL_{m+1}(\Of / p^{n+1} \Of))$.
\end{myex}

\subsection{Iwasawa algebras}
\label{Iwasawa}
Given any profinite group $G$, we can define the Iwasawa algebras $RG$ and $KG$ as follows.

\begin{mydef}
\begin{itemize}
\item $RG = \varprojlim_{N \unlhd_{o} G} R[G/N]$.

\item $KG = K \otimes_{R} RG$.

\end{itemize}

\end{mydef}

\noindent Now suppose $G$ is an $F$-uniform group. We can then describe the structure of $RG$ and $KG$ as follows. 
\begin{myprop}
\label{iwasawastructure}
Let $\{x_{1},\dots,x_{r}\}$ be a minimal topological generating set for $G$. Write $b^{s} = (x_{1}-1)^{s_{1}} \dots (x_{r} - 1)^{s_{r}} \in RG$ for $s \in \mathbb{N}_{0}^{r}$. Then $RG$ can be described as follows:

$$
RG = \{ \underset{s \in \mathbb{N}_{0}^{r}}{\sum} C_{s} b^{s} \mid C_{s} \in R \}
$$

\noindent It follows from this that we can describe $KG$ as:

$$
KG = \{ \underset{s \in \mathbb{N}_{0}^{r}}{\sum} C_{s} b^{s} \mid C_{s} \in K, \text{ the } C_{s} \text{ bounded as } |s|\rightarrow \infty \}
$$
\end{myprop}

\begin{proof}
A standard fact about Iwasawa algebras, for example \cite[\S 28]{Schneider}.
\end{proof}

\noindent Now assume $\mathfrak{g}$ is an $\Of$-Lie algebra such that $L_{G} = p^{n+1} \mathfrak{g}$. We wish to embed $KG$ in the completed enveloping algebra $\Ug$. In the setting of the introduction, this allows us to consider completed highest-weight modules for $\mathfrak{g}_{K}$ as $KG$-modules. To this end, we make the following claims inspired by \cite[Theorem 10.4]{irredreps}.

We write $\text{id}$ for the identity map $G \rightarrow L_{G} = p^{n+1} \mathfrak{g}$. This allows us to describe a mapping of $G$ into $\Ug$ as follows.

\begin{myprop}
The map $x \mapsto \exp(\text{id}(x))$ defines an injective continuous homomorphism from $G$ to the multiplicative group $\Ug^{\times}$ (with the subspace topology on $\Ug^{\times}$ given by the filtration on $\Ug$).
\end{myprop}

\begin{proof}
Note that for any $y \in p^{n+1} \mathfrak{g}$, $\exp(y) = \sum \frac{y^{j}}{j!}$ defines an element of $\Ug$. So $x \mapsto \exp(\text{id}(x))$ does define a map $G \rightarrow \Ug$  . Moreover, Theorem \ref{categoryisom} implies that $x * y = \log(\exp(x)\exp(y))$ defines an operation on $G$ that coincides with the group operation. This implies that $x \mapsto \exp(\text{id}(x))$ is a homomorphism. Moreover, it is clear that for any $y \in p^{n+1}\mathfrak{g} \setminus 0$ we have $\exp(y) \neq 1 $ in $\Ug$ and so the map is injective.

For the continuity, we choose a minimal topological generating set $x_{1},\dots,x_{r}$ for $G$. By \cite[Theorem 4.9]{prop} this choice gives a homeomorphism $\mathbb{Z}_{p}^{r} \rightarrow G$ via $a \mapsto x_{1}^{a_{1}} \dots x_{r}^{a_{r}}$. So we can consider the topology of $G$ as being given by the filtration $G_{i} = \{x_{1}^{a_{1}} \dots x_{r}^{a_{r}} \mid p^{i} | a_{j} \text{ for all } j \}$. Also note that $\exp(p^{n+1} \mathfrak{g}) \subseteq \Gamma_{1} \Ug$, so we can see that $\exp(G_{i}) \subseteq \Gamma_{rp^{i}} \Ug$. Therefore $x \mapsto \exp (x)$ is indeed continuous.
\end{proof}

\begin{mycor}

There is an embedding of $K$-algebras $KG \rightarrow \Ug$ given by $x \mapsto \exp(\text{id}(x))$ for $x \in G$.
\end{mycor}

\begin{proof}
Using \cite[Corollary 19.3]{Schneider}, the map $x  \mapsto \exp(\text{id}(x))$ extends to an $R$-homomorphism $RG \mapsto \Ug$, and hence to a $K$-homomorphism $\phi : KG \rightarrow \Ug$.

Now we check that $\phi$ is injective. To do this, we fix a $\mathbb{Z}_{p}$-basis $x_{1},\dots,x_{r}$ for $\mathfrak{g}$. Then $p^{n+1}x_{1},\dots, p^{n+1} x_{r}$ is a $\mathbb{Z}_{p}$-basis for $L_{G}$ and hence a topological generating set for $G$ by Proposition \ref{basisprop}. Then the elements $b^{s} = (p^{n+1}x_{1}-1)^{s_{1}} \dots (p^{n+1}x_{r}-1)^{s_{r}}$ (as in Proposition \ref{iwasawastructure}) satisfy:

\begin{equation*}
\begin{split}
\phi(b^{s}) & = (e^{p^{n+1}x_{1}}-1)^{s_{1}} \dots (e^{p^{n+1}x_{r}}-1)^{s_{r}} \\
 & = p^{(n+1)|s|}x_{1}^{s_{1}} \dots x_{r}^{s_{r}} + \text{ (higher degree terms)}.
\end{split}
\end{equation*}

Now suppose $0 \neq x = \underset{s \in \mathbb{N}_{0}^{r}}{\sum} C_{s} b^{s} \in KG$ satisfies $\phi(x) = 0$. Suppose $m \in \mathbb{N}_{0}$ is minimal such that some $s$ with $|s|=m$ satisfies $C_{s} \neq 0$. Then $0 = \phi(x) = \underset{|s|=m}{\sum} C_{s} x_{1}^{s_{1}} \dots x_{m}^{s_{m}} + \text{higher degree terms}$. Lemma \ref{PBW} now implies that all $C_{s}$ with $|s|=m$ are zero, a contradiction. So $\phi$ is injective.

\end{proof}

\noindent This proof now implies the following result.

\begin{mycor}
\label{embedding}
Given a $\mathbb{Z}_{p}$-basis $x_{1},\dots,x_{r}$ for $\mathfrak{g}$, we can identify $KG$ with the following subalgebra of $\Ug$:

$$
KG = \{ \sum_{s \in \mathbb{N}_{0}^{r}} C_{s} (e^{p^{n+1}x_{1}}-1)^{s_{1}} \dots (e^{p^{n+1}x_{r}}-1)^{s_{r}} \mid C_{s} \in K \text{ are bounded} \}.
$$

\end{mycor}

\subsection{A faithfulness result for Iwasawa algebras}

The following result is a modification of \cite[Theorem 5.3]{verma}.

\begin{mythm}
\label{5.3}
Let $G$ be an $F$-uniform group with associated $\Of$-Lie algebra $L_{G} = p^{n+1}\mathfrak{g}$. Let $N,P$ be $F$-uniform subgroups with associated $\Of$-Lie algebras $L_{N} = p^{n+1}\mathfrak{n}, L_{P} = p^{n+1}\mathfrak{p}$ respectively such that $\mathfrak{g}= \mathfrak{n} \oplus \mathfrak{p}$. 

Suppose $\hat{M}$ is a $\Ug$-module, such that $\hat{M}$ is free as a $\Un$-module by restriction with a basis $\mathcal{B}$, and that $\un \mathcal{B}$ is a faithful locally finite $RP$-module. Then $\hat{M}$ is faithful as a $KG$-module.
\end{mythm}

\begin{proof}

Let $g_{1},\dots,g_{r}$ be a minimal topological generating set for $G$ such that $g_{1},\dots,g_{s}$ is a topological generating set for $N$ and $g_{s+1},\dots,g_{n}$ is a topological generating set for $P$ (this can be done due to Proposition \ref{basisprop}).

Also let $x_{1},\dots,x_{m}$ be an $\Of$-basis for $\mathfrak{g}$ such that $x_{1},\dots,x_{i}$ and $x_{i+1}, \dots,x_{m}$ are bases for $\mathfrak{n}, \mathfrak{p}$ respectively. 

Write $b^{\alpha} = (g_{1}-1)^{\alpha_{1}} \dots (g_{r}-1)^{\alpha_{r}}$ for $\alpha \in \mathbb{N}_{0}^{r}$, and $x^{\beta} = x_{1}^{\beta_{1}} \dots x_{m}^{\beta_{m}}$ for $\beta \in \mathbb{N}_{0}^{m}$.

Let $\zeta \in \Ann_{RG} \hat{M}$. Then $\zeta$ has the form $\zeta = \underset{\alpha \in \mathbb{N}_{0}^{r}}{\sum} C_{\alpha} b^{\alpha}$ for some $C_{\alpha} \in R$. So we can write $\zeta = \underset{\alpha \in \mathbb{N}_{0}^{s}}{\sum} b^{\alpha} \zeta_{\alpha}$ for some $\zeta_{\alpha} \in RP$.

Let $v \in \un \mathcal{B}$. Then since $\un \mathcal{B}$ is locally finite as an $RP$-module, $RP v$ is contained in the $R$-span of a set $\pi^{-r} \{ x^{\beta} w \mid \beta \in \mathbb{N}_{0}^{i}, |\beta| \leq k, w \in \mathcal{B}' \}$, where $k \in \mathbb{N}$, $r \in \mathbb{N}$ and $\mathcal{B}'$ is a finite subset of $\mathcal{B}$.

So for each $\alpha \in \mathbb{N}_{0}^{s}$ we can write $\zeta_{\alpha} \cdot v = \underset{|\beta| \leq k, w \in \mathcal{B}' }{\sum} D_{\alpha,\beta,w} x^{\beta}w$ for some $D_{\alpha,\beta,w} \in R$, uniformly bounded in $\alpha, \beta$ and $w$. 

Then we have : 

\begin{equation*}
\begin{split}
0 & = \zeta \cdot v \\
& = \sum_{\alpha} b^{\alpha} \zeta^{\alpha} \cdot v \\
& = \sum_{\alpha} b^{\alpha} \sum_{\beta,w} D_{\alpha,\beta,w} x^{\beta}w \\ 
& = \sum_{w} \sum_{\alpha,\beta} D_{\alpha,\beta,w}b^{\alpha} x^{\beta} w. 
\end{split}
\end{equation*}

Since $\mathcal{B}$ is a free basis for $\hat{M}$ over $\Un$, we can see that $\underset{\alpha,\beta}{\sum} D_{\alpha,\beta,w}b^{\alpha}x^{\beta} = 0$ for each $w$. So now fix a $w \in \mathcal{B}$.

Now, from the proof of \cite[Theorem 5.3]{verma}, the multiplication map $KN \otimes_{K} \un \rightarrow \Un$ is injective. So $\underset{\alpha,\beta}{\sum} D_{\alpha,\beta,w} b^{\alpha} \otimes x^{\beta} = 0$. Since the $x^{\beta} \in \un$ are linearly independent, this means $\underset{\alpha}{\sum} D_{\alpha, \beta, w} b^{\alpha} = 0$ for any $\beta$, and so all $D_{\alpha,\beta,w}=0$ since the $b^{\alpha}$ are linearly independent in $KN$.

Thus $\zeta_{\alpha} \cdot v = 0$ for each $\alpha$ and each $v \in \un \mathcal{B}$. Since $\un \mathcal{B}$ is faithful over $RP$, it now follows that each $\zeta_{\alpha} = 0$. So $\zeta = 0$.

\end{proof}

\noindent The main setting in which we want to use this result is for induced modules, as in the following corollary.

\begin{mycor}
\label{5.3cor} 
Let $G$ be an $F$-uniform group with associated $\Of$-Lie algebra $L_{G} = p^{n+1}\mathfrak{g}$. Let $N,P$ be $F$-uniform subgroups with associated $\Of$-Lie algebras $L_{N} = p^{n+1}\mathfrak{n}, L_{P} = p^{n+1}\mathfrak{p}$ respectively such that $\mathfrak{g}= \mathfrak{n} \oplus \mathfrak{p}$. 

Let $V \in \mathcal{M}_{n}(\mathfrak{p})$ be finite-dimensional over $K$. Take $M = \ug \otimes_{\up} V$. Suppose $M$ is a faithful $RP$-module. Then $\hat{M}$ is a faithful $KG$-module.
\end{mycor}

\noindent Note that the condition that $V \in \mathcal{M}_{n}(\mathfrak{p})$ is for instance met if $\mathfrak{p}$ has a triangular decomposition and $V \in \mathcal{O}_{n}(\mathfrak{p})$ by Corollary \ref{catOincatM}. Also note that Proposition \ref{locallyfinite} implies that $M$ can indeed be considered as an $RP$-module.

\begin{proof}
Taking a $K$-basis $\mathcal{B}$ for $V$, $\mathcal{B}$ is a generating set for $M$ over $\ug$. This set is a free $\un$-basis for $M$, so we can define a corresponding good filtration $\Gamma_{i}M$ for $M$. Given a free $\Of$-basis $x_{1},\dots,x_{r}$ for $\mathfrak{n}$, $\Gamma_{0}M$ has free $R$-basis $\{ (p^{n}x)^{s} v \mid s \in \mathbb{N}_{0}^{r}, v \in \mathcal{B} \}$ where $(p^{n}x)^{s}$ denotes $(p^{n}x_{1})^{s_{1}} \dots (p^{n} x_{r})^{s_{r}}$.

So $\hat{M}$ can be described as $\{ \underset{s \in \mathbb{N}_{0}^{r}, v \in \mathcal{B}}{\sum} C_{s,v} (p^{n}x)^{s} v \mid C_{s,v} \rightarrow 0 \text{ as } |s| \rightarrow \infty \}$, and in particular is a free $\Un$-module with finite basis $\mathcal{B}$. Then $M = \un \mathcal{B}$ is a locally finite $RP$-module by Proposition \ref{locallyfinite}. So if $M$ is faithful over $RP$, then $\hat{M}$ is faithful over $KG$ by Theorem \ref{5.3}.

\end{proof}

\section{Reduction to the semisimple case}
\label{reductionsection}

\subsection{Set-up}

For this section, let $\mathfrak{g} = \mathfrak{g}_{s} \oplus \mathfrak{a}$ be an $\Of$-lattice of a split reductive $F$-Lie algebra with root system $\Phi$, a fixed set of simple roots $\Delta$, $\mathfrak{a}$ the centre of $\mathfrak{g}$, and a Chevalley basis  $\{e_{\alpha},f_{\alpha} \mid \alpha \in \Phi^{+} \} \cup \{ h_{\alpha} \mid \alpha \in \Delta \}$ for $\mathfrak{g}_{s}$. Say $\mathfrak{h}_{s}$ is the Cartan subalgebra of $\mathfrak{g}$, and $\mathfrak{h} = \mathfrak{a} \oplus \mathfrak{h}_{s}$.

Suppose $I \subseteq \Delta$ and $\mathfrak{h}' \subseteq \mathfrak{h}$ is a subalgebra such that $\mathfrak{h} = \mathfrak{h}' \oplus \mathfrak{a}$ and $h_{\alpha} \in \mathfrak{h}'$ for all $\alpha \in I$. Suppose $\lambda: \mathfrak{h} \to K$ is a weight such that $\lambda(h_{\alpha}) \in \mathbb{N}_{0}$ for all $\alpha \in I$. 

Let $\mathfrak{p}$ be the parabolic subalgebra of $\mathfrak{g}_{s}$ corresponding to $I$, and write $\mathfrak{p}' = \mathfrak{p}_{I,\mathfrak{h}'}$. As in \S \ref{gvm} there is a Lie algebra isomorphism $f : \mathfrak{p}' \rightarrow \mathfrak{p}$ given by the projection of $\mathfrak{g}$ onto $\mathfrak{g}_{s}$ with kernel $\mathfrak{a}$. This extends to a $K$-Lie algebra isomorphism $\mathfrak{p}'_{K} \rightarrow \mathfrak{p}_{K}$ which maps $p^{n} \mathfrak{p}'_{R}$ to $p^{n} \mathfrak{p}_{R}$. Using the universal property of universal enveloping algebras, this extends to a filtered ring isomorphism $\upp \rightarrow \up $, and hence to an isomorphism $f : \Upp \rightarrow \Up$.

Now define $\mu \in \mathfrak{h}_{K}^{*}$ by the composition $\mathfrak{h}_{K} = \mathfrak{h}'_{K} \oplus \mathfrak{a}_{K} \xrightarrow{\text{projection}} \mathfrak{h}'_{K} \xrightarrow{\lambda |_{\mathfrak{h}'_{K}}} K $. So $\mu(h) = \lambda(h)$ for $h \in \mathfrak{h}'$ and $\mu(h) = 0$ for $h \in \mathfrak{a}$. Let $V$ be the irreducible highest-weight $U(\mathfrak{p}_{K}'+\mathfrak{a}_{K})$-module of highest-weight $\mu$. Then $V$ restricts to both a $\up$-module and to a $\upp$-module. 

We now have a generalised Verma module for $\mathfrak{g}_{s,K}$ given by $M_{I}(\mu |_{\mathfrak{h}_{s}}) = U(\mathfrak{g}_{s,K}) \otimes_{\up} V$. Also note that $\mu |_{\mathfrak{h}'_{K}} = \lambda |_{\mathfrak{h}'_{K}}$ by definition of $\mu$, which means that the restriction of $V$ to $\upp$ can be considered as the irreducible highest-weight module of weight $\lambda |_{\mathfrak{h}'_{K}}$. Thus $M_{I,\mathfrak{h}'}(\lambda) = \ug \otimes_{\upp}V$. 

We then take the $F$-uniform groups $G,G_{s},P,P'$ corresponding to the $\Of$-Lie algebras $p^{n+1}\mathfrak{g},p^{n+1} \mathfrak{g}_{s},p^{n+1}\mathfrak{p},p^{n+1}\mathfrak{p}'$ respectively.

With this notation established, we can now state the main result of this section.

\begin{mythm}
\label{reductionprop}
Suppose $\widehat{M_{I}(\mu |_{\mathfrak{h}_{s}})}$ is faithful as a $KG_{s}$-module. Then $\widehat{M_{I,\mathfrak{h}'}(\lambda)}$ is faithful as a $KG$-module.
\end{mythm}

\noindent Note that by Corollary \ref{5.3cor} if $M_{I,\mathfrak{h}'}(\lambda)$ is faithful as an $RP'$-module, then $\widehat{M_{I,\mathfrak{h}'}(\lambda)}$ is faithful as a $KG$-module. So we will consider $M_{I,\mathfrak{h}'}(\lambda)$ as an $RP'$-module. We in fact only need the hypothesis that $M_{I}(\mu |_{\mathfrak{h}_{s}})$ is faithful as an $RP$-module.

The structure of the proof will be as follows.

\begin{itemize}
\item Show that $\Ann_{RP'} M_{I,\mathfrak{h}'}(\lambda) \subseteq \Ann_{RP'} \ug \otimes_{U(\mathfrak{p}'_{K} + \mathfrak{a}_{K})} V$. 

\item Show that $f(\Ann_{RP'} \ug \otimes_{U(\mathfrak{p}'_{K} + \mathfrak{a}_{K})} V) = \Ann_{RP} \ug \otimes_{U(\mathfrak{p}'_{K} + \mathfrak{a}_{K})} V$. In particular, $\ug \otimes_{U(\mathfrak{p}'_{K} + \mathfrak{a}_{K})} V$ is faithful over $RP$ if and only if it is faithful over $RP'$.

\item Show that  $M_{I}(\mu |_{\mathfrak{h}_{s}})$ is isomorphic to $\ug \otimes_{U(\mathfrak{p}'_{K} + \mathfrak{a}_{K})} V$ as a $U(\mathfrak{g}_{s,K})$-module. 

\item Conclude that if $M_{I}(\mu |_{\mathfrak{h}_{s}})$ is faithful over $RP$, then $M_{I,\mathfrak{h}'}(\lambda)$ is faithful over $RP'$.

\end{itemize}

\subsection{First step}

\begin{mylem}
$\ug \otimes_{U(\mathfrak{p}'_{K} + \mathfrak{a}_{K})} V$ is a $\ug$-quotient of $\ug \otimes_{\upp} V = M_{I,\mathfrak{h}'}(\lambda)$.
\end{mylem}

\begin{proof}

In general, suppose $A \leq B \leq C$ are rings and $N$ is a $B$-module. By the universal property of tensor products, there is a map $C \otimes_{A} N \rightarrow C \otimes_{B} N$ given by $c \otimes x \mapsto c \otimes x$, which is a surjective $C$-module homomorphism. Now apply this with $A = U(\mathfrak{p}'_{K}), B = U(\mathfrak{p}'_{K} + \mathfrak{a}_{K}), C = U(\mathfrak{g}_{K})$ and $N = V$.
\end{proof}

\begin{mycor}
\label{reduction3}
$\Ann_{RP'} M_{I,\mathfrak{h}'}(\lambda) \subseteq \Ann_{RP'} \ug \otimes_{U(\mathfrak{p}'_{K} + \mathfrak{a}_{K})} V$. 
\end{mycor}

\begin{proof}

Note that $M_{I,\mathfrak{h}'}(\lambda)$ and  $\ug \otimes_{U(\mathfrak{p}'_{K} + \mathfrak{a}_{K})} V$ are both locally finite $\upp$-modules by Proposition \ref{locallyfinite}, so can be considered as $\Upp$-modules by Proposition \ref{locallyfiniteprop}. 

The above lemma now implies that $\ug \otimes_{U(\mathfrak{p}'_{K} + \mathfrak{a}_{K})} V$ is a $\upp$-quotient of $M_{I,\mathfrak{h}'}(\lambda)$, and hence also a $\Upp$-quotient. Thus $\Ann_{\Upp} M_{I,\mathfrak{h}'}(\lambda) \subseteq \Ann_{\Upp} \ug \otimes_{U(\mathfrak{p}'_{K} + \mathfrak{a}_{K})} V$, and the result follows since $RP' \subseteq \Upp$.
\end{proof}

\subsection{The image of the annihilator under $f$}

Now we claim that $f : \Upp \rightarrow  \Up$ maps the $RP'$-annihilator of $\ug \otimes_{U(\mathfrak{p}'_{K} + \mathfrak{a}_{K})} V$ to the $RP$-annihilator. We will then look at the annihilator in $RP$.

\begin{mylem}
\begin{equation*}
f(\Ann_{\Upp} \ug \otimes_{U(\mathfrak{p}'_{K} + \mathfrak{a}_{K})} V) = \Ann_{\Up} \ug \otimes_{U(\mathfrak{p}'_{K} + \mathfrak{a}_{K})} V.
\end{equation*}
\end{mylem}

\begin{proof}
Let $x \in \mathfrak{g}$. Then $x - f(x) \in \mathfrak{a}$. Now, $\mathfrak{a}$ acts trivially on $V$ by choice of $\mu$, and hence acts trivially on $\ug \otimes_{U(\mathfrak{p}'_{K} + \mathfrak{a}_{K})} V$ since $\mathfrak{a}$ is central in $\mathfrak{g}$. Thus for any $v \in \ug \otimes_{U(\mathfrak{p}'_{K} + \mathfrak{a}_{K})} V$ we have $x v = f(x) v$.

Now consider the algebra isomorphism $f : \Upp \rightarrow \Up$. For any $x \in \Upp$, $x$ and $f(x)$ have the same action on $\ug \otimes_{U(\mathfrak{p}'_{K} + \mathfrak{a}_{K})} V$.
\end{proof}

\begin{mylem}
$f(RP') = RP$. 
\end{mylem}

\begin{proof}

Fix a $\mathbb{Z}_{p}$-basis $x_{1},\dots,x_{r}$ for $\mathfrak{p}'$. Then by Corollary $\ref{embedding}$, we can consider $RP'$ as the ring of power series in $e^{p^{n+1}x_{i}} - 1$, contained in $\Upp$. Now, $f (e^{p^{n+1}x_{i}}-1) = e^{p^{n+1}f(x_{i})} - 1$ with the $f(x_{i})$ forming a $\mathbb{Z}_{p}$-basis for $\mathfrak{p}$. So indeed $f(RP') = RP$.

\end{proof}

\noindent Combining the two previous results, we obtain the following.

\begin{mycor}
\label{reduction2}
$\ug \otimes_{U(\mathfrak{p}'_{K} + \mathfrak{a}_{K})} V$ is faithful over $RP'$ if and only if it is faithful over $RP$.
\end{mycor}

\subsection{Proof of Theorem \ref{reductionprop}}

There is now only one lemma left in order to prove Theorem \ref{reductionprop}.

\begin{mylem}
\label{reduction1}
$M_{I}(\mu |_{\mathfrak{h}_{s}})$ is isomorphic to $\ug \otimes_{U(\mathfrak{p}'_{K} + \mathfrak{a}_{K})} V$ as a $U(\mathfrak{g}_{s,K})$-module.
\end{mylem}

\begin{proof}
Take a $K$-basis $\mathcal{B}$ for $V$, and let $x_{1},\dots,x_{k}$ be the elements of $\{f_{\alpha} \mid \alpha \in \Phi^{+} \setminus \Phi_{I} \}$ in some order. Then both $M_{I}(\mu |_{\mathfrak{h}_{s}}) = U(\mathfrak{g}_{s,K}) \otimes_{\up} V$ and $\ug \otimes_{U(\mathfrak{p}'_{K} + \mathfrak{a}_{K})} V$ have $K$-bases $\{ x^{s} \otimes v \mid s \in \mathbb{N}_{0}^{k}, v \in \mathcal{B} \}$. 

Now, there is naturally a $U(\mathfrak{g}_{s,K})$-homomorphism given by the composition $U(\mathfrak{g}_{s,K}) \otimes_{\up} V \hookrightarrow  \ug \otimes_{\up} V \twoheadrightarrow \ug \otimes_{U(\mathfrak{p}_{K}+\mathfrak{a}_{K})} V =  \ug \otimes_{U(\mathfrak{p'}_{K}+\mathfrak{a}_{K})} V$, which is an isomorphism because it maps each $x^{s} \otimes v$ in $M_{I}(\mu |_{\mathfrak{h}_{s}})$ to $x^{s} \otimes v$ in $\ug \otimes_{U(\mathfrak{p'}_{K}+\mathfrak{a}_{K})} V$.
\end{proof}

\noindent We now combine these pieces for a proof of Theorem \ref{reductionprop}.

\begin{proof}[Proof of Theorem \ref{reductionprop}]
Suppose $\widehat{M_{I}(\mu |_{\mathfrak{h}_{s}})}$ is faithful as a $KG_{s}$-module. Note that Lemma \ref{densitylem} applies to $M_{I}(\mu |_{\mathfrak{h}_{s}})$ since it lies in category $\mathcal{O}_{n}(\mathfrak{g})$. So any $x \in KG_{s}$ such that $x \cdot M_{I}(\mu |_{\mathfrak{h}_{s}})  = 0$ in $ \widehat{M_{I}(\mu |_{\mathfrak{h}_{s}})}$ also annihilates $\widehat{M_{I}(\mu |_{\mathfrak{h}_{s}})}$, and hence must equal zero. So $M_{I}(\mu |_{\mathfrak{h}_{s}})$ is faithful as an $RP$-module.

By Lemma \ref{reduction1}, $\ug \otimes_{U(\mathfrak{p}'_{K} + \mathfrak{a}_{K})} V$ is faithful as an $RP$-module. Then by Corollary \ref{reduction2}, $\ug \otimes_{U(\mathfrak{p}'_{K} + \mathfrak{a}_{K})} V$ is faithful as an $RP'$-module.

By Corollary \ref{reduction3}, $M_{I,\mathfrak{h}'}(\lambda)$ is faithful as an $RP'$-module. Finally, by Corollary \ref{5.3cor} $\widehat{M_{I,\mathfrak{h}'}(\lambda)}$ is faithful over $KG$.

\end{proof}

\section{The duality functor on category $\mathcal{O}$}
\label{dualsection}

\subsection{Defining the duality functor}
We now restrict to the case where $\mathfrak{g}$ is semisimple, with Chevalley basis $\{e_{\alpha},f_{\alpha} \mid \alpha \in \Phi^{+} \} \cup \{h_{\alpha} \mid \alpha \in \Delta \}$. The goal of this section is to prove the following result.

\begin{mythm}
\label{leviprop}
Suppose $\hat{M} = \widehat{M_{I}(\lambda)}$ is an  affinoid generalised Verma module such that $I \subset \Delta$ is totally proper, and $M = M_{I}(\lambda)$. Let $\mathfrak{l} = \mathfrak{l}_{I}$ be the corresponding Levi subalgebra of $\mathfrak{g}$, with $L$ being the $F$-uniform group corresponding to $p^{n+1}\mathfrak{l}$.

If $M$ is faithful over $RL$, then $\hat{M}$ is faithful over $KG$.
\end{mythm}

\noindent Recall our definitions of the categories $\mathcal{O}(\mathfrak{g})$ and $\mathcal{O}_{n}(\mathfrak{g})$ from Definition \ref{categoryO}. We will now define the duality functor on category $\mathcal{O}(\mathfrak{g})$ as in \cite{Humphreys}.

\begin{mydef}

Define the transpose map $\tau : \mathfrak{g} \rightarrow \mathfrak{g}$ by $\tau(e_{\alpha}) = f_{\alpha}, \tau(f_{\alpha}) = e_{\alpha}$ for all $\alpha \in \Phi^{+}$, and $\tau(h_{\alpha}) = h_{\alpha}$ for all $\alpha \in \Delta$. Then by the relations assumed on our Chevalley basis, $\tau$ is a Lie algebra anti-automorphism that maps $\mathfrak{g}_{R} \rightarrow \mathfrak{g}_{R}$. Thus $\tau$ extends to a filtered anti-automorphism $\tau: \ug \rightarrow \ug$, and hence to a filtered anti-automorphism $\tau : \Ug \rightarrow \Ug$.

\end{mydef}

\begin{mydef}
For $M \in \mathcal{O}(\mathfrak{g})$, let $M^{*}$ be the dual vector space of $M$, with $\ug$-action $(x \cdot \phi)(v) = \phi(\tau(x) v)$ for $x \in \ug, \phi \in M^{*}, v \in M$.

Define $M^{\lor}$ to be the subspace of $M^{*}$ consisting of all $\phi \in M^*$ such that $\phi(M_{\lambda}) = 0$ for all but finitely many $\lambda$. Then $M^{\lor}$ is the \emph{dual module} of $M$.

\end{mydef}

\noindent We now state some facts from \cite{Humphreys} about the duality functor.

\begin{myprop}
\label{dualityfacts}
\begin{itemize}
\item The functor $^{\lor} : \mathcal{O}(\mathfrak{g}) \rightarrow \mathcal{O}(\mathfrak{g})$ is an exact contravariant functor.

\item For any $M \in \mathcal{O}(\mathfrak{g})$, $M$ and $M^{\lor}$ have the same composition factor multiplicities.

\item For any $M \in \mathcal{O}(\mathfrak{g})$ and any weight $\lambda$, $\dim_{K} M_{\lambda} = \dim_{K} M_{\lambda}^{\lor}$. In particular, $^{\lor}$ restricts to a functor $\mathcal{O}_{n}(\mathfrak{g}) \rightarrow \mathcal{O}_{n}(\mathfrak{g})$.

\end{itemize}

\begin{proof}
The first two parts are \cite[Theorem 3.2]{Humphreys}. For the third part, $\dim_{K} M_{\lambda} = \dim_{K} M_{\lambda}^{\lor}$ because $M_{\lambda}^{\lor}$ is isomorphic as a $K$-vector space to the standard dual space $(M_{\lambda})^{*}$.

Suppose $M \in \mathcal{O}_{n}(\mathfrak{g})$ and $\lambda$ is a weight of $M^{\lor}$. Then $\dim M_{\lambda} = \dim M_{\lambda}^{\lor} \neq 0$, so $\lambda$ is a weight of $M$ and thus $\lambda(p^{n} \mathfrak{h}_{R}) \subseteq R$. So $M^{\lor} \in \mathcal{O}_{n}(\mathfrak{g})$. 
\end{proof}

\end{myprop}

\subsection{Completions and the duality functor}

We now want to see how the duality functor interacts with our completions.

\begin{mylem}

Suppose $M \in \mathcal{O}_{n}(\mathfrak{g})$ and $x \in \Ug$ is a weight vector. Then for any $\phi \in M^{\lor}$, $x \cdot \phi \in M^{\lor}$ is given by $(x \cdot \phi)(v) = \phi(\tau(x) v)$ for $v \in M$.
\end{mylem}

\begin{proof}
Firstly, note that $x \cdot \phi$ can be considered as an element of  $M^{\lor}$ by Lemma \ref{weightvectorlem}.

Both sides of the equality are linear in $\phi$, so we can assume $\phi$ is a weight vector. Say $\phi$ has weight $\mu$ and $x$ has weight $\nu$, then $x \cdot \phi$ has weight $\mu + \nu$. So for $v$ in any weight space of $M$ other than $M_{ \mu + \nu}$, we have $(x \cdot \phi)(v) = 0 = \phi(\tau(x) v)$. So it is sufficient to check the equality for $v \in M_{\mu + \nu}$. Fix some $v \in M_{\mu + \nu}$.

Fix an ordering of $\Phi^{+}$, and write $e^{s} = \underset{\alpha \in \Phi^{+}}{\prod} e_{\alpha}^{s_{\alpha}}$ and $ f^{s} = \\underset{\alpha \in \Phi^{+}}{\prod} f_{\alpha}^{s_{\alpha}}$ for $s \in \mathbb{N}_{0}^{  \Phi^{+}}$ (with the order of the multiplication subject to the chosen ordering of $\Phi^{+}$). Then using Lemma \ref{PBW} we can write $x$ in the form $\underset{s,t}{\sum} f^{t} e^{s} H_{s,t}$ where $H_{s,t} \in \Uh$ and the sum is over all $s,t \in \mathbb{N}_{0}^{\Phi}$ such that $f^{t} e^{s}$ has weight $\nu$. Then $\tau(x) = \underset{s,t}{\sum}H_{s,t}f^{s}e^{t} $.

Since $M^{\lor}$ is locally finite over $U(\mathfrak{n}_{K}^{+})$, for all but finitely many $s$ we have $e^{s} \cdot \phi = 0$. Moreover, for each fixed $s$ there are only finitely many $t$ such that $f^{t} e^{s}$ has weight $\nu$. Therefore there are only finitely many pairs $s,t$ such that $f^{t} e^{s}$ has weight $\nu$ and $f^{t} e^{s} \cdot \phi \neq 0$.

Similarly, since $M$ is locally finite over $U(\mathfrak{n}_{K}^{+})$, there are only finitely many pairs $s,t$ such that $f^{t} e^{s}$ has weight $\nu$ and $f^{s} e^{t} \cdot v \neq 0$.

So let $\mathcal{S}$ be the set of all the pairs $(s,t)$ with $s,t \in \mathbb{N}_{0}^{\Phi^{+}}$ such that $f^{t} e^{s}$ has weight $\nu$ and either $f^{t} e^{s} \cdot \phi \neq 0$ or $f^{s} e^{t} \cdot v \neq 0$. Then $\mathcal{S}$ is a finite set. 

So $(x \cdot \phi)(v) = \underset{(s,t) \in \mathcal{S}}{\sum} (f^{t}e^{s}H_{s,t} \cdot \phi)(v) = \underset{(s,t) \in \mathcal{S}}{\sum} \mu(H_{s,t}) (f^{t} e^{s} \cdot \phi)(v)$ using Lemma \ref{diagonalmult}.

Now, whenever $(s,t) \in \mathcal{S}$, we know $f^{t} e^{s}$ has weight $\nu$. Thus $f^{s}e^{t}$ has weight $-\nu$, and so $f^{s} e^{t} v$ has weight $\mu$. So we have:

\begin{equation*}
\begin{split}
(x \cdot \phi)(v) & = \sum_{(s,t) \in \mathcal{S}} \mu(H_{s,t}) (f^{t} e^{s} \cdot \phi)(v) \\
& = \sum_{(s,t) \in \mathcal{S}} \mu(H_{s,t}) \phi(f^{s}e^{t}v)\\
& = \sum_{(s,t) \in \mathcal{S}} \phi(H_{s,t} \cdot f^{s} e^{t} v)\\
&= \phi( \sum_{(s,t) \in \mathcal{S}} H_{s,t} f^{s} e^{t} v)\\
&= \phi(\tau(x) \cdot v)
\end{split}
\end{equation*}
  
\end{proof}

\begin{mycor}
\label{dualannihilatorcor}
Suppose $x \in \Ann_{\Ug} \hat{M}$ is a weight vector. Then $\tau(x)  \cdot M^{\lor} = 0$ in $\widehat{M^{\lor}}$.
\end{mycor}

\begin{proof}
By the previous lemma, for any $\phi \in M^{\lor}$ and $v \in M$ we have $(\tau(x) \cdot \phi)(v) = \phi(x v) = \phi(0)=0$. So $x$ annihilates $M^{\lor}$.

\end{proof}

\subsection{Proof of Theorem \ref{leviprop}}

Now consider the setting of Theorem \ref{leviprop}. That is, we have $I \subseteq \Delta$ and $\lambda \in \mathfrak{h}_{K}^{*}$ such that $\lambda(p^{n}\mathfrak{h}_{R}) \subseteq \mathfrak{h}_{R}$ and $\lambda(h_{\alpha}) \in \mathbb{N}_{0}$ for $\alpha \in I$. Now take $M = M_{I}(\lambda)$. Let $\mathfrak{p} = \mathfrak{p}_{I}$ and let $\mathfrak{p}^{-} = \tau(\mathfrak{p})$ be the opposite parabolic subalgebra, with $F$-uniform group $P^{-}$ corresponding to $p^{n+1} \mathfrak{p}^{-}$. Note that $M \in \mathcal{O}_{n}(\mathfrak{g})$.

\begin{mylem}
$\tau : \Ug \rightarrow \Ug$ satisfies $\tau(KP) = KP^{-}$.
\end{mylem}

\begin{proof}
Suppose $x_{1}, \dots, x_{r}$ is a $\mathbb{Z}_{p}$-basis for $\mathfrak{p}$. Then $\tau(x_{1}), \dots, \tau(x_{r})$ is a $\mathbb{Z}_{p}$-basis for $\mathfrak{p}^{-}$. We now use the description of the embedding $KP \subseteq \Up$ from Corollary \ref{embedding}. Given $x \in KP$, we can write:

$$
x = \sum_{s \in \mathbb{N}_{0}^{r}} C_{s} (e^{p^{n+1}x_{1}}-1)^{s_{1}} \dots (e^{p^{n+1}x_{r}}-1)^{s_{r}}
$$

\noindent with the $C_{s}$ bounded. Then:

$$
\tau(x) = \sum_{s \in \mathbb{N}_{0}^{r}} C_{s} (e^{p^{n+1}\tau(x_{r})}-1)^{s_{r}} \dots (e^{p^{n+1}\tau(x_{1})}-1)^{s_{1}}
$$

\noindent We can now see that the image of $KP$ under $\tau$ is precisely $KP^{-}$.
\end{proof}

\noindent The following result allows us to apply Corollary \ref{5.3cor} to $M$ considered as a $KP^{-}$-module. We decompose $\mathfrak{p}^{-}$ as $\mathfrak{p}^{-} = \mathfrak{l} \oplus \mathfrak{n}$ where $\mathfrak{n}$ is spanned by all $f_{\alpha}$ with $\alpha \in \Phi^{+} \setminus \Phi_{I}^{+}$. 

\begin{mylem}
$M$ is isomorphic as a $U(\mathfrak{p}_{K}^{-})$-module to $U(\mathfrak{p}_{K}^{-}) \otimes_{\ul} V$, where $V$ is the irreducible highest-weight $\up$-module of weight $\lambda$ and $V$ is finite-dimensional.

\end{mylem}

\begin{proof}
Given a $K$-basis $\mathcal{B}$ for $V$ and an ordering $\{f_{1},\dots,f_{r} \}$ of all $f_{\alpha} \in \mathfrak{n}$, the PBW theorem implies that both $M$ and $U(\mathfrak{p}_{K}^{-}) \otimes_{\ul} V$ have a $K$-basis $\{ f_{1}^{s_{1}} \dots f_{r}^{s_{r}} \otimes v \mid s \in \mathbb{N}_{0}^{r}, v \in \mathcal{B} \}$.

Now, there is naturally a $U(\mathfrak{p}^{-}_{K})$-homomorphism given by the composition $U(\mathfrak{p}^{-}_{K}) \otimes_{\ul} V \rightarrow \ug \otimes_{\ul} V \rightarrow \ug \otimes_{\up} V = M$, which is an isomorphism since it sends $f_{1}^{s_{1}} \dots f_{r}^{s_{r}} \otimes v $ in $U(\mathfrak{p}_{K}^{-}) \otimes_{\ul} V$ to 
$f_{1}^{s_{1}} \dots f_{r}^{s_{r}} \otimes v $ in $M$

The fact that $V$ is finite-dimensional follows from the condition $\lambda(h_{\alpha}) \in \mathbb{N}_{0}$ for all $\alpha \in I$.
\end{proof}

\begin{proof}[Proof of Theorem \ref{leviprop}]

Suppose $M$ is faithful over $RL$. Then by the above lemma, we can apply Corollary $\ref{5.3cor}$ to see that $M$ is also faithful as a $KP^{-}$-module. Now we wish to apply our results with the duality functor to see that $M$ is faithful over $RP$.

Suppose $x \in \Ann_{RP}M$. Say $x = \underset{\lambda}{\sum} x_{\lambda}$ as a convergent sum, where $x_{\lambda} \in \Up$ is the $\lambda$-
weight component of $x$. Then $\tau(x) = \underset{\lambda}{\sum} \tau(x_{\lambda})$ as a convergent sum of weight components. Moreover, $\tau(x) \in RP^{-}$.

For each $\lambda$, $x_{\lambda}$ annihilates $M$ using Lemma \ref{weightcompslem} (and the fact $M \in \mathcal{O}_{n}(\mathfrak{g})$). Therefore $\tau(x_{\lambda})$ annihilates $M^{\lor}$ by Corollary \ref{dualannihilatorcor}. We know by Proposition \ref{dualityfacts} that $M$ and $M^{\lor}$ have the same composition factors as $\ug$-modules, so $\tau(x_{\lambda})$ annihilates all composition factors of $M$.

The fact that $M \in \mathcal{O}(\mathfrak{g})$ implies $M$ has finite composition length by \cite[Theorem 1.11]{Humphreys}.
Let $l$ be the composition length of $M$.  So for any weights $\lambda_{1},\dots,\lambda_{l}$ we now know that $\tau(x_{\lambda_{1}}) \dots \tau(x_{\lambda_{l}})$ annihilates $M \subseteq \hat{M}$. Therefore $\tau(x)^{l} \in RP^{-} $ annihilates $M \subseteq \hat{M}$. By Lemma \ref{densitylem}, $\tau(x)^{l} \in \Ann_{KP^{-}} \hat{M}$.

Since $\hat{M}$ is faithful over $KP^{-}$, we now see that $\tau(x)^{l} = 0$. We know $KP^{-}$ is a domain, for instance since it embeds in $\widehat{U(\mathfrak{g})_{n,K}}$, so $\tau(x) = 0$. Since $\tau$ is injective, $x = 0$.

Therefore $M$ is faithful over $RP$. By Corollary \ref{5.3cor}, $\hat{M}$ is now faithful over $KG$.

\end{proof}

\section{Annihilator of a collection of finite-dimensional modules}
\label{annsection}

\subsection{Motivation and outline}
Motivated by Theorem \ref{leviprop}, we want to show that an affinoid generalised Verma module $M_{I}(\lambda)$ is faithful over the Iwasawa algebra $RL$, where $L$ corresponds to $p^{n+1} \mathfrak{l}$ with $\mathfrak{l} = \mathfrak{l}_{I}$.

To this end, we look at a family of infinitely many finite-dimensional $\ul$-submodules of $M_{I}(\lambda)$. Note that any finite-dimensional $RL$-module is certainly not faithful. But we want to show that the simultaneous annihilator in $RL$ of our family of finite-dimensional submodules is contained in the annihilator of some induced module for $\mathfrak{l}$.

More specifically, the family of finite-dimensional modules we look at will be the irreducible highest-weight $\ul$-modules $L(\lambda - \underset{\alpha \in \Delta \setminus I}{\sum} c_{\alpha} \alpha)$ for $c_{\alpha} \in \mathbb{N}_{0}$. 

Denote $L_{\lambda,c} = L(\lambda - \underset{\alpha \in \Delta \setminus I}{\sum} c_{\alpha} \alpha)$ for any $c \in (p^{-n}R)^{\Delta \setminus I}$. Write $I^{o} = \{ \beta \in I \mid \alpha(h_{\beta}) = 0 \text{ for all } \alpha \in \Delta \setminus I \}$, which we can consider as the interior of $I$ in the Dynkin diagram with nodes $\Delta$. Also write $M_{I,\lambda,c}$ for the generalised Verma module  $M_{I^{o}}(\lambda - \sum c_{\alpha} \alpha)$. Note that for any $\beta \in I^{o}$ and any $c$ we have $(\lambda - \sum c_{\alpha} \alpha)(h_{\beta}) = \lambda(h_{\beta}) \in \mathbb{N}_{0}$, so we can define such a generalised Verma module.

The main points of our argument will then be as follows.

\begin{enumerate}
\item We will use a density argument to show that if $x \in \Ul$ annihilates all $L_{\lambda,c}$ for $c_{\alpha} \in \mathbb{N}_{0}$, then $x$ annihilates all $L_{\lambda,c}$ for $c_{\alpha} \in p^{-n}R$. 

\item We will define surjective $\Ul$-homomorphisms $\phi_{c} : \widehat{M_{I,\mathfrak{h}'}(\lambda)} \rightarrow \widehat{M_{I,\lambda,c}}$ for $c \in (p^{-n}R)^{\Delta \setminus I}$. Then if $x \in \underset{c \in \mathbb{N}_{0}^{\Delta \setminus I}}{\bigcap} \Ann_{\Ul} L_{\lambda,c}$ and $v \in \widehat{M_{I,\mathfrak{h}'}(\lambda)}$, we can see that $\phi_{c}(x \cdot v) = x \phi_{c}(v) = 0$ for all $c$ such that $M_{I,\lambda,c}$ is irreducible.

\item We show that if $w \in \widehat{M_{I,\mathfrak{h}'}(\lambda)}$  and $\phi_{c}(w) = 0$ for all $c$ such that $M_{I,\lambda,c}$ is irreducible, then $w = 0$. This will involve looking at a certain set of $c$ for which an irreducibility criterion is met, then using a density argument.

\end{enumerate}

\subsection{Set-up}

In this section, we slightly change notation to make clearer how the main result of the section relates to Theorem $\ref{mainthm}$. That is, $\mathfrak{g}$ now plays the role of $\mathfrak{l}$ from \S \ref{dualsection}, so is the Levi subalgebra of a larger semisimple Lie algebra.

 We now assume that $\mathfrak{g}_{\Delta}$ is an $\Of$-lattice of a split semisimple $F$-Lie algebra with root system $\Phi$ and that $\mathfrak{g}_{\Delta}$ has a Chevalley basis, and that $\mathfrak{g}$ is the Levi subalgebra corresponding to the totally proper subset $I \subseteq \Delta$. So $\mathfrak{g}$ is a reductive Lie algebra with root system $\Phi_{I}$. As before, write $\mathfrak{g}_{s}$ for the semisimple part of $\mathfrak{g}$ and $\mathfrak{h}$ for the Cartan subalgebra. Let $\lambda \in \mathfrak{h}_{K}^{*}$ be a weight such that $\lambda(p^{n} \mathfrak{h}_{R}) \subseteq R$ and such that $\lambda(h_{\alpha}) \in \mathbb{N}_{0}$ for all $\alpha \in I$ (meaning $\lambda$ restricts to a dominant integral weight for $\mathfrak{g}_{s,K}$). Also write $\Delta \setminus I = \{\alpha_{1},\dots,\alpha_{r} \}$.

As before, use the notation $I^{o} = \{ \alpha \in I \mid \alpha_{i}(h_{\alpha}) = 0 \text{ for all } i \}$. For $c \in (p^{-n}R)^{r}$ write $L_{\lambda,c} = L(\lambda - \sum c_{i} \alpha_{i})$ and $M_{I,\lambda,c} = M_{I^{o}}(\lambda - \sum c_{i} \alpha_{i})$.

In this setting, we can use the following facts.

\begin{itemize}
\item $\alpha_{i}(h_{\alpha}) \in \mathbb{Z}_{\leq 0}$ for all $1 \leq i \leq r$ and $\alpha \in I$.

\item The $K$-basis $\mathcal{H}$ for $\mathfrak{h}_{K}$ dual to $\Delta$ is an $\Of$-basis for $\mathfrak{h}$, using the assumption from the introduction that $p$ does not divide the determinant of the Cartan matrix of $\Phi$. (In fact, this is the reason behind our assumption on $p$).

\end{itemize}

We now note that given a weight $\mu \in \mathfrak{h}_{K}^{*}$ and a highest-weight $U(\mathfrak{g}_{s,K})$-module $M$ with highest-weight $\mu |_{\mathfrak{h}_{s}}$, we can consider $M$ as a highest-weight $\ug$-module with highest-weight $\mu$. We do this by inflating  along the homomorphism $U(\mathfrak{g}_{K}) \cong U(\mathfrak{g}_{s,K}) \otimes_{K} U(\mathfrak{a}_{K}) \rightarrow U(\mathfrak{g}_{s,K})$ given by mapping $h \in \mathfrak{a}_{K}$ to $\mu(h)$. 

So we write $L(\mu)$ for the irreducible highest-weight $U(\mathfrak{g}_{s,K})$-module of weight $\mu$ extended to a $\ug$-module, and if $J \subseteq I$ such that $\mu(h_{\alpha}) \in \mathbb{N}_{0}$ for $\alpha \in J$ we write $M_{J}(\mu)$ for the generalised Verma module for $U(\mathfrak{g}_{s,K})$ extended to a $\ug$-module. In particular, if in addition $\mu(p^{n}\mathfrak{h}_{R}) \subseteq R$ then these modules lie in $\mathcal{O}_{n}(\mathfrak{g})$ and so we can take their completions.

Now consider the set of weights $C = \{ \lambda - \underset{i}{\sum} c_{i} \alpha_{i} \mid c_{i} \in \mathbb{N}_{0} \}$. Note that for any choice of $c_{i} \in \mathbb{N}_{0}$ and any $\alpha \in I$, we have $\lambda(h_{\alpha}) - \underset{i}{\sum} c_{i} \alpha_{i}(h_{\alpha}) \in \mathbb{N}_{0}$, meaning that the irreducible highest-weight $\ug$-module $L_{\lambda,c} = 
L(\lambda - \underset{i}{\sum} c_{i} \alpha_{i})$ is finite-dimensional. In this section we will look at the simultaneous annihilator in $\Ug$ of all the finite-dimensional modules $L(\mu), \mu \in C$. In particular, we will aim to show this is contained in the annihilator of some induced module for $\mathfrak{g}$.

To this end, we need a subalgebra of $\mathfrak{g}$ and a finite-dimensional module in order to obtain an induced module. Take $\mathfrak{h}'$ to be spanned by the elements of $\mathcal{H}$ corresponding to $I$. That is, $\mathfrak{h}'$ has $\Of$-basis $\{h'_{\alpha} \mid \alpha \in I \}$ satisfying $\beta (h'_{\alpha}) = \delta_{\alpha,\beta}$ for $\beta \in \Delta$. 
Then $\mathfrak{h}'$ contains all $h_{\alpha}$ with $\alpha \in I^{o}$.

Using the notation of \S \ref{gvm} we now take $\mathfrak{p} = \mathfrak{p}_{I^{o},\mathfrak{h}'}$. Take $V$ to be the irreducible highest-weight $\mathfrak{p}_{K}$-module of weight $\lambda |_{\mathfrak{h}'_{K}}$. Then $V$ is finite-dimensional over $K$. The induced module we now wish to study is $M_{I^{o},\mathfrak{h}'}(\lambda) = \ug \otimes_{\up} V$. The main result of this section is as follows.

\begin{mythm}
\label{annthm}
$\underset{\mu \in C}{\bigcap} \Ann_{\Ug}L(\mu) \subseteq \Ann_{\Ug} \widehat{(M_{I^{o},\mathfrak{h}'}(\lambda))}$.
\end{mythm}

\subsection{A density argument}

The result we aim to prove in this subsection is as follows. For any $c \in (p^{-n}R)^{r}$, use the notation $L_{\lambda,c} = L(\lambda - \underset{i}{\sum} c_{i} \alpha_{i})$.

\begin{myprop}
\label{annprop}
Suppose $x \in \underset{\mu \in C}\bigcap \Ann_{\Ug}L(\mu)$. Then for any $c \in (p^{-n}R)^{r}$, we have $x \in \Ann_{\Ug} \widehat{L_{\lambda,c}}$.
\end{myprop}

\noindent To prove this, we first want the following lemma. Write $v_{c}$ for the highest-weight vector of $L_{\lambda,c}$ for any $c \in (p^{-n}R)^{r}$.

\begin{mylem}
Suppose $x \in \Ug$ has weight $0$ and $x \cdot v_{c} = 0$ in $L_{\lambda,c}$ for all $c \in \mathbb{N}_{0}^{r}$. Then $x \cdot v_{c} = 0$ in $\widehat{L_{\lambda,c}}$ for all $c \in (p^{-n}R)^{r}$.
\end{mylem}

\begin{proof}
By Lemma \ref{PBW} and the fact that $\lambda(p^{n}R) \subseteq R$, we can write $x$ in the form:

\begin{equation*}
x = \sum_{s,t,u} b_{s,t,u} \prod_{\alpha \in \Phi_{I}^{+}} f_{\alpha}^{s_{\alpha}} \prod_{\alpha \in \Phi_{I}^{+}} e_{\alpha}^{t_{\alpha}} \prod_{h \in \mathcal{H}}  h^{u_{h}} 
\end{equation*}

\noindent where $s,t$ run over $\mathbb{N}_{0}^{\Phi_{I}^{+}}$ and $u$ runs over $\mathbb{N}_{0}^{\mathcal{H}}$, where the product is subject to a fixed ordering of $\Phi_{I}^{+}$, and where the convergence condition $p^{-n(|s|+|t|+|u|)}b_{s,t,u} \rightarrow 0$ as $|s|+|t|+|u| \rightarrow \infty$ is met. But then using the fact that $p^{n}\lambda(h) \in R$ for any $h \in \mathcal{H}$, we can instead write:

\begin{equation*}
x = \sum_{s,t,u} a_{s,t,u} \prod_{\alpha \in \Phi_{I}^{+}} f_{\alpha}^{s_{\alpha}} \prod_{\alpha \in \Phi_{I}^{+}} e_{\alpha}^{t_{\alpha}} \prod_{h \in \mathcal{H}}  (\lambda(h) - h)^{u_{h}} 
\end{equation*}

\noindent where $p^{-n(|s|+|t|+|u|)}a_{s,t,u} \rightarrow 0$ as $|s|+|t|+|u| \rightarrow \infty$.

Now note that since $x$ has weight $0$, any such monomial for which some $s_{\alpha}\neq 0$ also has some $t_{\alpha} \neq 0$ and hence annihilates all $v_{c}$. Also note that for $h \in \mathcal{H} \setminus \{h'_{\alpha_{1}},\dots,h'_{\alpha_{r}} \}$, and any $c \in (p^{-n}R)^{r}$, we have $(\lambda - \underset{i}{\sum} c_{i} \alpha_{i})(h) = \lambda(h)$ by choice of the basis $\mathcal{H}$, and so any monomial with $u_{h} \neq 0$ also annihilates $v_{c}$.

So we now write $W$ for the $K$-vector space of all convergent sums of the  $\underset{\alpha \in \Phi_{I}^{+}}{\prod} f_{\alpha}^{s_{\alpha}} \underset{\alpha \in \Phi_{I}^{+}}{\prod} e_{\alpha}^{t_{\alpha}} \underset{h \in \mathcal{H}}{\prod}  (\lambda(h) - h)^{u_{h}}$ such that either some $s_{\alpha} \neq 0$ or some $u_{h} \neq 0$ with $h \notin \{h'_{\alpha_{1}},\dots,h'_{\alpha_{r}} \}$. Then $W$ annihilates $v_{c}$ in $\widehat{L_{\lambda,c}}$ for any $c$.

We now write $x$ in the form $x = Q(\lambda(h'_{\alpha_{1}}) - h'_{\alpha_{1}},\dots, \lambda(h'_{\alpha_{r}}) - h'_{\alpha_{r}}) + (\text{element of } W)$, where $Q \in K \llbracket X_{1},\dots,X_{r} \rrbracket$ such that $Q$ converges on $(p^{-n}R)^{r}$.

Then for any $c \in p^{-n}R$ we have $x \cdot v_{c} = Q(c) v_{c}$, since $h'_{\alpha_{i}} \cdot v_{c} = \lambda(h'_{\alpha_{i}}) v_{c} - \underset{j}{\sum}c_{j} \alpha_{j}(h'_{\alpha_{i}}) = \lambda(h'_{\alpha_{i}}) - c_{i}$ and so $(\lambda(h'_{\alpha_{i}}) - h'_{\alpha_{i}})\cdot v_{c} = c_{i} v_{c}$.

Now, $Q(c) = 0$ for all $c \in \mathbb{N}_{0}^{r}$, and also $Q$ converges on $R$. So \cite[Lemma 4.7]{verma} implies that $Q = 0$. Thus $x \cdot v_{c} = 0$ for all $c \in (p^{-n}R)^{r}$.
\end{proof}

\begin{proof}[Proof of Proposition \ref{annprop}]
First note that by Lemma \ref{weightcompslem}, all weight vectors of $x$ lie in $\underset{\mu \in C}{\bigcap}\Ann_{\Ug}L(\mu)$. So it is sufficient to prove the result in the case where $x$ is a weight vector. So from now on, assume $x$ is a weight vector.

Suppose $c \in (p^{-n}R)^{r}$ such that $x \cdot \widehat{L_{\lambda,c}} \neq 0$. Then by Lemma \ref{densitylem}, $x$ does not annihilate $L_{\lambda,c}$. Therefore $x$ does not annihilate all weight vectors of $L_{\lambda,c}$.

So choose a weight vector $v \in L_{\lambda,c}$ such that $xv \neq 0$. Then say $v = y v_{c}$, where $y \in \ug$ is a weight vector.

Now since $L_{\lambda,c}$ is irreducible, we can choose a weight vector $z \in \ug$ such that $z x y v_{c} = v_{c}$. Thus $z x y$ has weight $0$. Moreover, $z x y \in \underset{\mu \in A}{\bigcap} \Ann_{\Ug}L(\mu)$ since each $\Ann_{\Ug}L(\mu)$ is a two-sided ideal. But then the above lemma implies that $z x y v_{c} = 0$, a contradiction.
\end{proof}

\subsection{Projections from $M_{I,\mathfrak{h}'}(\lambda)$}

First, we describe the structure of $M_{I,\mathfrak{h}'}(\lambda)$ as follows. Recall $V$ is the irreducible highest-weight $\up$-module of weight $\lambda$, which arises from the irreducible highest-weight $U(\mathfrak{g}_{I^{o},K})$-module of weight $\lambda |_{\mathfrak{h}_{I^{o}}}$. Let $w$ be the highest-weight vector of $V$, and let $\mathcal{B}$ be a free $R$-basis for $U(p^{n}\mathfrak{g}_{I^{o},R})w$ (which exists by Proposition \ref{catOstructure}).

 We write $f^{s} = \prod f_{\alpha}^{s_{\alpha}}$ for $s \in \mathbb{N}_{0}^{\Phi_{I}^{+}\setminus \Phi_{I^{o}}^{+}}$, and write $h^{t} = \prod (h'_{\alpha_{i}})^{t_{i}}$ for $t \in \mathbb{N}_{0}^{r}$, subject to a choice of ordering of $\Phi_{I}^{+}\setminus \Phi_{I^{o}}^{+}$. Then $M_{I,\mathfrak{h}'}(\lambda)$ has a $K$-basis $f^s h^t v $ for $s \in \mathbb{N}_{0}^{\Phi_{I}^{+}\setminus \Phi_{I^{o}}^{+}}, t \in \mathbb{N}_{0}^{r}$ and $v \in \mathcal{B}$.

Now note that $U(p^{n}\mathfrak{g}_{R})w$ has an $R$-basis consisting of the $p^{n(|s|+|t|)}f^{s}h^{t}v$ with $s \in \mathbb{N}_{0}^{\Phi_{I}^{+}\setminus \Phi_{I^{o}}^{+}}, t \in \mathbb{N}_{0}^{r}$ and $v \in \mathcal{B}$. So we use the good filtration $\Gamma_{i} M_{I,\mathfrak{h}'}(\lambda) = \pi^{i} R \{ p^{n(|s|+|t|)}f^{s}h^{t}v \mid s \in \mathbb{N}_{0}^{\Phi_{I}^{+}\setminus \Phi_{I^{o}}^{+}}, t \in \mathbb{N}_{0}^{r},v \in \mathcal{B} \}$ which satisfies $(M_{I,\mathfrak{h}'}(\lambda),\Gamma) \in \mathcal{M}_{n}(\mathfrak{g})$.

Now fix some $c \in (p^{-n}R)^{r}$ and consider the structure of $M_{I,\lambda,c} = 
M_{I}(\lambda - \underset{i}{\sum} c_{i} \alpha_{i})$. Note first that $(\lambda - \underset{i}{\sum} c_{i} \alpha_{i}) |_{\mathfrak{h}'} = \lambda |_{\mathfrak{h}'}$. Therefore we can (as in our construction of generalised Verma modules in \S \ref{gvm}) extend $V$ to an irreducible highest-weight $U(\mathfrak{p}'_{K}+\mathfrak{h}_{K})$ with highest-weight $\lambda - \sum c_{i} \alpha_{i}$. Given this, we can restrict $V$ to a $U(\mathfrak{p}_{I^{o},K})$-module (where $\mathfrak{p}_{I^{o}}$ is the parabolic subalgebra of $\mathfrak{g}_{s}$ corresponding to $I^{o}$). Then $V |_{U(\mathfrak{p}_{I^{o},K})}$ is the irreducible highest-weight module of weight $(\lambda - \underset{i}{\sum} c_{i} \alpha_{i}) |_{\mathfrak{h}_{s}}$. We can thus write $M_{I,\lambda,c} = U(\mathfrak{g}_{s,K}) \otimes_{U(\mathfrak{p}_{I^{o},K})} V $, extended to a $\ug$-module.

Therefore $M_{I,\lambda,c}$ has a $K$-basis $f^{s} v$ for $s \in \mathbb{N}_{0}^{\Phi_{I}^{+}\setminus \Phi_{I^{o}}^{+}}$ and $v \in \mathcal{B}$. We take the good filtration $\Gamma_{i}M_{I,\lambda,c} = \pi^{i} R \{p^{n|s|}f^{s} v \mid s \in \mathbb{N}_{0}^{\Phi_{I}^{+}\setminus \Phi_{I^{o}}^{+}}, v \in \mathcal{B} \}$, which gives $(M_{I,\lambda,c},\Gamma) \in \mathcal{M}_{n}(\mathfrak{g})$. 

\begin{mylem}
The embedding of $V$ in $M_{I,\lambda,c}$ as $U(\mathfrak{p}_{I^{o},K})$-modules is also an embedding of $U(\mathfrak{p}_{I^{o},K} + \mathfrak{h}_{K})$-modules.
\end{mylem}

\begin{proof}
We need to check that any $h \in \mathfrak{h}_{K}$ has the same action on $V$ and the image of $V$ in $M_{I,\lambda,c} = U(\mathfrak{g}_{s,K}) \otimes_{U(\mathfrak{p}_{I^{o},K})} V$. 

Suppose we write $v$ for the highest-weight vector generating $V$, and write $w$ for its image in $M_{I,\lambda,c}$. Then by definition $V$ is spanned by elements $f_{\beta_{1}} \dots f_{\beta_{l}} v$ for $\beta_{i} \in  \Phi_{I^{o}}^{+}$. Then any $h \in \mathfrak{h}_{K}$ acts on both $f_{\beta_{1}} \dots f_{\beta_{l}} v$ and $f_{\beta_{1}} \dots f_{\beta_{l}} w$ by the scalar $(\lambda - \sum c_{i} \alpha_{i} - \sum \beta_{i})(h)$.
\end{proof}

\noindent In particular, we can see that the embedding of $V$ into $M_{I,\lambda,c}$ is an embedding of $U(\mathfrak{p}_{K})$-modules.

\begin{myprop}
Fix $c \in (p^{-n}R)^{r}$. There is a filtered $\ug$-homomorphism $\phi_{c}: M_{I^{o},\mathfrak{h}'}(\lambda) \rightarrow M_{I,\lambda,c}$ given by mapping $f^{s} h^{t} v$ to $f^{s} (\lambda(h)-c)^{t} v$ for $s \in \mathbb{N}_{0}^{\Phi_{I}^{+}\setminus \Phi_{I^{o}}^{+}}, t \in \mathbb{N}_{0}^{r}$ and $v \in \mathcal{B}$, where $(\lambda(h)-c)^{t}$ denotes $\prod (\lambda(h'_{\alpha_{i}}) - c_{i})^{t_{i}}$.

\end{myprop}

\begin{proof}
We use the embedding of $V$ in $M_{I,\lambda,c}$ as $\up$-modules. By the universal property of tensor products, we can define a homomorphism of abelian groups  $\phi_{c} : M_{I^{o},\mathfrak{h}'}(\lambda) = \ug \otimes_{\up} V \rightarrow M_{I,\lambda,c}$ by $x \otimes v \mapsto x \cdot v$ for $x \in \ug, v \in V$.  Then $\phi_{c}$ is also a $\ug$-module homomorphism by definition.

Now for any $j$, since $h'_{\alpha_{j}}$ commutes with $U(\mathfrak{g}_{I^{o},K})$, we can see that $h'_{\alpha_{j}}$ acts on $V$ by the scalar $(\lambda - \sum c_{i} \alpha_{i})(h'_{\alpha_{j}}) = \lambda(h'_{\alpha_{j}}) - c_{j}$. Therefore we can see that $\phi_{c}$ maps $f^{s} h^{t} v$ to $f^{s} (\lambda(h)-c)^{t} v$ as required.

We can now see the homomorphism $\phi_{c}$ is a filtered homomorphism because $p^{n(|s|+|t|)}f^{s}h^{t}v \in \Gamma_{0}M_{I,\mathfrak{h}'}(\lambda)$ maps to $p^{n|t|}(\lambda(h)-c)^{t} p^{n|s|}f^{s} v$, which lies in $\Gamma_{0}M_{I,\lambda,c}$ because each $p^{n}c_{i}$ and $p^{n} \lambda(h'_{\alpha_{i}})$ lies in $R$.

\end{proof}

\begin{mycor}
For any $c \in (p^{-n}R)^{r}$ there is a filtered $\Ug$-homomorphism $\phi_{c}: \widehat{M_{I^{o},\mathfrak{h}'}(\lambda)} \rightarrow \widehat{M_{I,\lambda,c}}$ given by mapping $f^{s} h^{t} v$ to $f^{s} c^{t} v$ for $s \in \mathbb{N}_{0}^{\Phi_{I}^{+}\setminus \Phi_{I^{o}}^{+}}, t \in \mathbb{N}_{0}^{r}$ and $v \in \mathcal{B}$.
\end{mycor}

\noindent It is now helpful to consider a set of $c$ such that $M_{I,\lambda,c}$ is irreducible, so that we can use Proposition \ref{annprop}. In fact this module is generically irreducible, but we will just find a set of such $c$ that is sufficiently large for our purposes. To this end, choose $A \in \mathbb{N}$ to be sufficiently large that $\langle \lambda + \rho, \alpha^{\lor} \rangle - A \notin \mathbb{N}$ for all $\alpha \in \Phi_{I}$ (where $\rho = \frac{1}{2} \underset{\alpha \in \Phi^{+}}{\sum}$).

\begin{mylem}
\label{gvmirredlem}
Suppose $c \in \mathbb{Z}_{\leq -A}^{r}$. Then $M_{I,\lambda,c}$ is irreducible, so is isomorphic to $ L_{\lambda,c}$.
\end{mylem}

\begin{proof}
By \cite[Theorem 9.12]{Humphreys} it is sufficient to show that $\langle \lambda - \sum c_{i} \alpha_{i} + \rho, \alpha^{\lor} \rangle \notin \mathbb{N}$ for all $\alpha \in \Phi_{I}^{+} \setminus \Phi_{I^{o}}^{+}$. We have $\langle \lambda - \sum c_{i} \alpha_{i} + \rho, \alpha^{\lor} \rangle = \langle \lambda + \rho, \alpha^{\lor} \rangle - \underset{i}{\sum} c_{i} \langle \alpha_{i},\alpha^{\lor} \rangle$. So fix $\alpha \in \Phi_{I}^{+} \setminus \Phi_{I^{o}}^{+}$.

Note that $\langle \alpha_{i},\alpha^{\lor} \rangle \in \mathbb{Z}_{\leq 0}$ for all $i$. Therefore if any $\langle \alpha_{i}, \alpha^{\lor} \rangle \neq 0$ we will have $\sum c_{i} \langle \alpha_{i},\alpha^{\lor} \rangle \in \mathbb{Z}_{\geq A}$, and hence  $\langle \lambda - \sum c_{i} \alpha_{i} + \rho, \alpha^{\lor} \rangle \notin \mathbb{N}$ by choice of $A$. So we now aim to show that for some $i$, $\langle \alpha_{i}, \alpha^{\lor} \rangle \neq 0$.

Say $\alpha = \underset{\gamma \in I}{\sum} a_{\gamma} \gamma$ where $a_{\gamma} \in \mathbb{N}_{0}$. Then since $\alpha \notin \Phi_{I^{o}}$, fix $\beta \in I \setminus I^{o}$ such that $a_{\beta} \neq 0$. Then fix $i$ such that $\langle \beta,\alpha_{i}^{\lor} \rangle \neq 0$. 

Now $\langle \alpha, \alpha_{i}^{\lor} \rangle = \underset{\gamma \in I}{\sum} a_{\gamma} \langle \gamma, \alpha_{i}^{\lor} \rangle$. Note that $a_{\gamma} \langle \gamma, \alpha_{i}^{\lor} \rangle \in \mathbb{Z}_{\leq 0}$ for all $\gamma$, and $a_{\beta} \langle \beta, \alpha_{i}^{\lor} \rangle < 0$. Thus $\langle \alpha, \alpha_{i}^{\lor} \rangle \in \mathbb{Z}_{<0}$, so $\langle \alpha_{i}, \alpha^{\lor}\rangle \neq 0$ as required
\end{proof}

\subsection{Proof of Theorem \ref{annthm}}
Fix $x \in \underset{\mu \in C}{\bigcap} \Ann_{\Ug}L(\mu)$, then we  will show that $x \in \Ann_{\Ug}\widehat{M_{I,\mathfrak{h}'}(\lambda)}$. So fix $v \in \widehat{M_{I,\mathfrak{h}'}(\lambda)}$, then we wish to show that $x \cdot v = 0$.

To show this, we first aim to show that $\phi_{c}(x \cdot v) = 0$ whenever $c \in \mathbb{Z}_{\leq - A}^{r}$. Note that $M_{I,\lambda,c} = L_{\lambda,c}$ by Lemma \ref{gvmirredlem}. Therefore we can consider $\phi_{c}(v)$ as an element of $L_{\lambda,c}$. Thus $x \cdot \phi_{c}(v) = 0$ by Proposition \ref{annprop}, and since $\phi_{c}$ is a $\Ug$-homomorphism, this means $\phi_{c}(x \cdot v) = x \cdot \phi_{c}(v) = 0$. The following lemma will now imply that $x \cdot v = 0$.

\begin{mylem}
Suppose $w \in \widehat{M_{I^{o},\mathfrak{h}'}(\lambda)}$ such that $\phi_{c}(w) = 0$ for all $c \in \mathbb{Z}_{\leq A}^{r}$. Then $w = 0$.
\end{mylem}

\begin{proof}
Using our basis $f^{s}h^{t}v$ for $s \in \mathbb{N}_{0}^{\Phi_{I} \setminus \Phi_{I^{o}}}, t \in \mathbb{N}_{0}^{r}, v \in \mathcal{B}$, write $w$ in the form $w = \sum a_{s,t,v} f^{s}h^{t}v$ where $a_{s,t,v} \in K$ and $p^{-(|s|+|t|)n}a_{s,t,v} \rightarrow 0$ as $|s|+|t| \rightarrow \infty$.

Then for $c \in \mathbb{Z}_{\leq -A}^{r}$, we have $0 = \phi_{c}(w) = \underset{s,t,v}{\sum}a_{s,t,v}(\lambda(h)-c)^{t}f^{s}v$. Therefore for any fixed $s,v$ we have $\underset{t}{\sum} a_{s,t,v} (\lambda(h)-c)^{t} = 0$. Now writing $Q_{s,v}(X) = \underset{t}{\sum} a_{s,t,v} X^{t}$ a formal power series which converges on $(p^{-n}R)^{r}$, we can see that $Q_{s,v} = 0$ on $\underset{i}{\prod} (\lambda(h'_{\alpha_{i}}) - \mathbb{Z}_{\leq_{-A}}) \subseteq (p^{-n}R)^{r}$. By applying \cite[Lemma 4.7]{verma}, we see that $Q_{s,v}(X) = 0$. Therefore $a_{s,t,v} = 0$ for all $s,t,v$, and hence $w = 0$. 

\end{proof}

\section{Proof of main theorem}
We now aim to prove Theorem \ref{mainthm}, which will also imply theorem \ref{theorem1}. Proceed by induction on the size of the set $I$. The base case is as follows.

\begin{mylem}
Theorem \ref{mainthm} is true when $I = \emptyset$.
\end{mylem}

\begin{proof}
By Theorem \ref{reductionprop} we can assume that $\mathfrak{g}$ is semisimple. Then the case where $I = \emptyset$ is the case of a Verma module, and is proved in \cite[Theorem 5.4]{verma}.
\end{proof}

Now we carry out the induction step. Again, we can assume that $\mathfrak{g}$ is semisimple by Theorem $\ref{reductionprop}$. So we have $\emptyset \neq I \subseteq \Delta$ not containing any connected component of $\Delta$, and $\lambda : p^{n} \mathfrak{h}_{R} \rightarrow R$ (extended $K$-linearly to $\mathfrak{h}_{K}$) such that $\lambda(\alpha) \in \mathbb{N}_{0}$ for all $\alpha \in I$. We aim to show that the affinoid generalised Verma module $\widehat{M_{I}(\lambda)}$ is faithful for $KG$.

Let $\mathfrak{l} = \mathfrak{l}_{I}$, and $L$ the $F$-uniform group with corresponding $\Of$-Lie algebra $p^{n+1} \mathfrak{l}$. By Theorem \ref{leviprop}, it is sufficient to show that $M_{I}(\lambda)$ is faithful for $RL$. To this end, we wish to apply Theorem \ref{annthm} to the reductive Lie algebra $\mathfrak{l}$. 

Writing $\Delta \setminus I = \{ \alpha_{1},\dots,\alpha_{r} \}$, taking $I^{o} = \{ \alpha \in I \mid \langle \alpha, \beta^{\lor} \rangle = 0 \text{ for all } \beta \in \Delta \setminus I \}$ and taking $\mathfrak{h}'$ to be the orthogonal complement of $\{h_{\alpha} \mid \alpha \in \Delta \setminus I \}$, Theorem \ref{annthm} implies that $ \underset{s \in \mathbb{N}_{0}^{r}}{\bigcap} \Ann_{RL} L(\lambda - \underset{i}{\sum}{s_{i} \alpha_{i}}) \subseteq \Ann_{RL} \widehat{M_{I^{o},\mathfrak{h}'}(\lambda)}$, where $M_{I^{o},\mathfrak{h}'}(\lambda)$ is the $\ul$-module as in \S \ref{gvm}. To use this, we want to show that $\Ann_{RL} M_{I}(\lambda) \subseteq \underset{s \in \mathbb{N}_{0}^{r}}{\bigcap} \Ann_{RL} L(\lambda - \underset{i}{\sum}{s_{i} \alpha_{i}})$.

Write $v_{0}$ for the highest-weight vector generating $M_{I}(\lambda)$.
Note that for any $s_{1},\dots,s_{r} \in \mathbb{N}_{0}$, $M_{I}(\lambda)$ contains an element $f_{\alpha_{1}}^{s_{1}} \dots f_{\alpha_{r}}^{s_{r}} \cdot v_{0}$, which is a $\ul$-highest-weight vector of weight $\lambda - \sum s_{i} \alpha_{i}$. This holds because any $e_{\alpha}$ with $\alpha \in \Phi_{I}$ commutes with the $f_{\alpha_{i}}$, noting that $\alpha - \alpha_{i}$ cannot be a root since it has positive coefficients of some simple roots in $I$ and a negative $\alpha_{i}$-coefficient. Thus $\Ann_{RL} M_{I}(\lambda) \subseteq \cap_{s \in \mathbb{N}_{0}^{r}} \Ann_{RL} L(\lambda - \sum{s_{i} \alpha_{i}})$.

Now if we can show that $I^{o}$ is a totally proper subset of $I$, the induction hypothesis will give that $\Ann_{RL} \widehat{M_{I^{o},\mathfrak{h}'}(\lambda)} = 0$ since then $|I^{o}| < |I|$.

Combining these facts, we see that $\Ann_{RL} M_{I}(\lambda) \subseteq \cap_{s \in \mathbb{N}_{0}^{r}} \Ann_{RL} L(\lambda - \sum{s_{i} \alpha_{i}}) \subseteq \Ann_{RL} \widehat{M_{I^{o},\mathfrak{h}'}(\lambda)} = 0$. 

\begin{mylem}
$I^{o}$ is a totally proper subset of $I$.
\end{mylem}

\begin{proof}
Suppose $I_{1}$ is a connected component of $I$ contained in $I^{o}$. We then claim $I_{1}$ is a connected component of $\Delta$, which contradicts the assumption that $I$ contains no connected components of $\Delta$.

Let $\alpha \in I_{1} \subseteq I^{o}, \beta \in \Delta \setminus I_{1}$. We claim $\langle \alpha, \beta^{\lor} \rangle = 0$. If $\beta \notin I$, this holds by definition of $I^{o}$. If $\beta \in I$, this holds by assumption that $I_{1}$ is a connected component of $I$.
\end{proof}

\section{The case of $\mathfrak{sl}_{3}$}
\label{sl3case}
In the case where $\mathfrak{g} = \mathfrak{sl}_{3}(\Of)$, we can give a positive answer to \cite[Question B]{verma}. In order to show this, we use our result for generalised Verma modules together with the following technical result. 

For the moment, we can take $\mathfrak{g}$ to be an $\Of$-lattice of a split semisimple $\Of$-Lie algebra with root system $\Phi$ and a Chevalley basis as before. We use the notation $s_{\alpha}$ for the reflection corresponding to a root $\alpha \in \Phi$, and $\rho = \frac{1}{2} \underset{\alpha \in \Phi^{+}}{\sum} \alpha$. Given any $\lambda : \mathfrak{h}_{K} \rightarrow K$, we write $L(\lambda)$ for the irreducible highest-weight $\ug$-module of highest-weight $\lambda$. 

\begin{myprop}
\label{reflectionprop}
Let $\mathfrak{g}$ be an $\Of$-lattice of a split semisimple $\Of$-Lie algebra with root system $\Phi$ and such that $\mathfrak{g}$ has a Chevalley basis. Let $\lambda : \mathfrak{h}_{K} \rightarrow K$ be a weight such that $\lambda( p^{n} \mathfrak{h}_{R}) \subseteq R$. Suppose $\alpha \in \Delta$ such that $\langle \lambda + \rho, \alpha^{\lor} \rangle \notin \mathbb{N}_{0}$.

Then $\Ann_{\Ug} \widehat{L(\lambda)} \subseteq \Ann_{\Ug} \widehat{L(s_{\alpha} \cdot \lambda)}$.
\end{myprop}

\noindent Firstly, we prove the following lemma. Write $v_{\lambda}, v_{s_{\alpha} \cdot \lambda}$ for the highest-weight vectors of $M(\lambda)$ and $M(s_{\alpha} \cdot \lambda)$ respectively.

\begin{mylem}
Suppose $x \in \Ann_{\Ug} \widehat{L(\lambda)}$ has weight $0$. Then $x \cdot v_{s_{\alpha} \cdot \lambda} = 0$ in $\widehat{M(s_{\alpha} \cdot \lambda)}$.
\end{mylem}

\begin{proof}
By our assumption (from the introduction) that $p$ does not divide the determinant of the Cartan matrix of $\Phi$, we have an $\Of$-basis $\{h^{\beta} \mid \beta \in \Delta \}$ dual to $\Delta$, that is, with $\beta (h^{\gamma}) = \delta_{\beta,\gamma}$ for $\beta, \gamma \in \Delta$.

We consider the action of $x$ on $f_{\alpha}^{[s]} v_{\lambda}$, where $f_{\alpha}^{[s]} = \frac{1}{s!}f_{\alpha}^{s}$ for $s  \in \mathbb{N}_{0}$. We write $a = \lambda (h_{\alpha}) \notin \mathbb{N}_{0}$.

\medskip

\noindent \emph{Claim:} for any $s \in \mathbb{N}_{0}$, we have $e_{\alpha} \cdot f_{\alpha}^{[s]} v_{\lambda} = (a+1-s) f_{\alpha}^{[s-1]} v_{\lambda}$.

\noindent \emph{Proof of claim:}
\begin{equation*}
\begin{split}
e_{\alpha} \cdot f_{\alpha}^{[s]} v_{\lambda} & = \frac{1}{s!} \text{ad}(e_{\alpha})(f_{\alpha}^{s}) v_{\lambda} \\
& = \frac{1}{s!} \sum_{i=0}^{s-1}f_{\alpha}^{s-1-i} \text{ad}(e_{\alpha})(f_{\alpha}) f_{\alpha}^{i} v_{\lambda} \\
& =  \frac{1}{s!} \sum_{i=0}^{s-1} f_{\alpha}^{s-1-i} h_{\alpha} f_{\alpha}^{i} v_{\lambda} \\
& = \frac{1}{s!}\sum_{i=0}^{s-1} f_{\alpha}^{s-1-i} (\lambda - i \alpha)(h_{\alpha}) f_{\alpha}^{i} v_{\lambda} \text{ (because } f_{\alpha}^{i} v_{\lambda} \text{ has weight } \lambda - i \alpha \text{)}\\ 
& = \frac{1}{s!}\sum_{i=0}^{s-1} (a - 2i) f_{\alpha}^{s-1} v_{\lambda} \\
& = \frac{1}{s!} (sa - s(s-1)) f_{\alpha}^{s-1} v_{\lambda} \left( \text{ because } \sum_{i=0}^{s-1} i = \frac{1}{2} s(s-1) \right) \\
& =  \frac{1}{(s-1)!}(a+1-s) f_{\alpha}^{s-1} v_{\lambda}\\
& = (a+1-s)f_{\alpha}^{[s-1]} v_{\lambda}.
\end{split}
\end{equation*}

\medskip

As a consequence of this claim, we can see that $e_{\alpha}^{s} f_{\alpha}^{[s]} v_{\lambda} = (a+1-s)(a+2-s) \dots a v_{\lambda}$ is a non-zero multiple of $v_{\lambda}$, which means that the image of $f_{\alpha}^{[s]}$ in $L(\lambda)$ is non-zero.

Now apply Lemma \ref{PBW} and the fact that $\lambda(p^{n} \mathfrak{h}_{R}) \subseteq R$ to see that we can write $x$ as a convergent sum of monomials of the form: 

\begin{equation*}
C_{s,t,u} f_{\alpha}^{s_{\alpha}} e_{\alpha}^{t_{\alpha}} \prod_{\beta \in \Phi^{+} \setminus \{\alpha \}} f_{\beta}^{s_{\beta}} \prod_{\beta \in \Phi^{+} \setminus \{\alpha \}} e_{\beta}^{t_{\beta}} \prod_{\beta \in \Delta} (\lambda(h^{\beta}) - h^{\beta})^{u_{\beta}}
\end{equation*}

\noindent with  $s_{\beta},t_{\beta}, u_{\beta} \in \mathbb{N}_{0}, C_{s,t,u} \in K$, such that $p^{-|s|-|t|-|u|}C_{s,t,u} \rightarrow 0 $ as $|s|+|t|+|u| \rightarrow 0$.

Let $J$ be the set of all such monomials such that $t_{\beta} > 0 $  or $u_{\beta} > 0$ for some $\beta \neq \alpha$. Then we want to show that $f_{\alpha}^{[s]} v_{\lambda}$ and $v_{s_{\alpha} \cdot \lambda}$ are annihilated by the monomials in $J$.

To see this, firstly note that if $\beta \neq \alpha$, then $\lambda(h^{\beta}) - h^{\beta}$ commutes with $f_{\alpha}$ by definition of $h^{\beta}$. Thus $\lambda(h^{\beta}) - h^{\beta}$ annihilates all $f_{\alpha}^{[s]}v_{\lambda}$. Moreover, $(s_{\alpha} \cdot \lambda)(h^{\beta}) = (\lambda - \langle \lambda + \rho, \alpha^{\lor} \rangle \alpha)(h^{\beta}) = \lambda(h^{\beta})$, therefore $\lambda(h^{\beta}) - h^{\beta}$ annihilates $v_{s_{\alpha} \cdot \lambda}$. So this covers the case where $u_{\beta} \neq 0 $ for some $\beta \neq \alpha$.

Now claim that if $\beta \in \Phi^{+} \setminus \alpha$ then $e_{\beta}$ annihilates both all $f_{\alpha}^{[s]} v_{\lambda}$ and $v_{s_{\alpha} \cdot \lambda}$. It is clear that $e_{\beta} \cdot v_{s_{\alpha} \cdot \lambda} = 0$ because $v_{s_{\alpha} \cdot \lambda} $ is a highest-weight vector. We can also see that $e_{\beta} \cdot f_{\alpha}^{[s]} v_{\lambda} = 0$ because this element of $M(\lambda)$
has weight $\lambda - s \alpha + \beta \notin \lambda - \mathbb{N}_{0} \Delta$, which is not the weight of any non-zero element of $M(\lambda)$. So this covers the case where $t_{\beta} \neq 0$ for some $\beta \neq \alpha$.

Using the fact that $x$ has weight $0$, we can write $x = \underset{t \geq 0}{\sum} f_{\alpha}^{t} e_{\alpha}^{t} Q_{t}(\lambda(h^{\alpha}) - h^{\alpha}) + (\text{ monomials in } J)$ where the $Q_{t}$ are power series that converge on $p^{-n}R$. Then $x \cdot f_{\alpha}^{[s]}v_{\lambda} = \underset{t}{\sum} \left( \underset{i=0}{\overset{t-1}{\prod}}(s-i) \underset{i=0}{\overset{t}{\prod}} (a+i-s) \right) Q_{t}(s) f_{\alpha}^{[s]} v_{\lambda} $ for all $s \in \mathbb{N}_{0}$. Since $x \cdot L(\lambda) = 0$ and the images of $f_{\alpha}^{[s]} v_{\lambda}$ in $L(\lambda)$ is non-zero, this means that $\underset{t}{\sum} \left( \underset{i=0}{\overset{t-1}{\prod}}(s-i) \underset{i=0}{\overset{t}{\prod}} (a+i-s) \right)Q_{t}(s)  = 0$ for all $s \in \mathbb{N}_{0}$.

The power series $Q(X) = \underset{t}{\sum} \left( \underset{i=0}{\overset{t-1}{\prod}}(X-i) \underset{i=0}{\overset{t}{\prod}} (a+i-X) \right) Q_{t}(X)$ is convergent on $R$ since $x \in \Ug$, and is zero on all non-negative integers. Thus $Q = 0$ by \cite[Lemma 4.7]{verma}. So $0 = Q(a+1) = Q_{0}(a+1)$. 

Finally, since $s_{\alpha} \cdot \lambda = \lambda - (a+1) \alpha$, we have $x \cdot v_{s_{\alpha} \cdot \lambda} = Q_{0}(a+1) v_{s_{\alpha} \cdot \lambda} = 0$.
\end{proof}

\begin{proof}[Proof of Proposition \ref{reflectionprop}]
First, suppose $x \in \Ann_{\Ug} \widehat{L(\lambda)}$ is a weight vector, say with weight $\mu$. We claim that $x$ annihilates $L(s_{\alpha} \cdot \lambda)$ inside $\widehat{L(s_{\alpha} \cdot \lambda)}$, and hence by Lemma \ref{densitylem} annihilates $\widehat{L(s_{\alpha} \cdot \lambda)}$.

Suppose $v \in L(s_{\alpha} \cdot \lambda)$ is a weight vector such that $x \cdot v \neq 0$. Then say $y \in \ug$ is a weight vector such that $v$ is the image of $y \cdot v_{s_{\alpha} \cdot \lambda}$ in $L(s_{\alpha} \cdot \lambda)$. Then $xy \cdot v_{s_{\alpha} \cdot \lambda}$ is not in the maximal submodule of $M(s_{\alpha} \cdot \lambda)$, so we can choose a weight vector $z \in \ug$ such that $zxy$ has weight $0$ and $zxy \cdot v_{s_{\alpha} \cdot \lambda} = v_{s_{\alpha} \cdot \lambda}$. But then $zxy \in \Ann_{\Ug} \widehat{L(\lambda)}$ since $\Ann_{\Ug} \widehat{L(\lambda)}$ is a two-sided ideal, and this contradicts the previous lemma.

So $x$ annihilates all weight vectors of $ L(s_{\alpha} \cdot \lambda)$, and hence all of $ L(s_{\alpha} \cdot \lambda)$, hence also annihilates $\widehat{ L(s_{\alpha} \cdot \lambda)}$.

Now suppose we take any $x \in \Ann_{\Ug} \widehat{L(\lambda)}$. Then the weight components of $x$ also lie in $\Ann_{\Ug} \widehat{L(\lambda)}$ by Lemma \ref{weightcompslem}, and so also lie in $\Ann_{\Ug} \widehat{ L(s_{\alpha} \cdot \lambda)}$. Thus $x \in \Ann_{\Ug} \widehat{ L(s_{\alpha} \cdot \lambda)}$. 

\end{proof}

\noindent We also use the following irreducibility criterion for generalised Verma modules, which is due to Jantzen. Suppose $I \subseteq \Delta$ and $\lambda : \mathfrak{h}_{K} \rightarrow K$ such that $\lambda(h_{\alpha}) \in \mathbb{N}_{0}$ for all $\alpha \in I$. Define the following sets of roots:

\begin{itemize}
\item $\Psi = \Phi \setminus \Phi_{I}$,
\item $\Psi_{\lambda}^{+} = \{ \beta \in \Psi^{+} \mid \langle \lambda + \rho, \beta^{\lor} \rangle \in \mathbb{N} \}$,
\item $\Phi_{\beta} = (\mathbb{Q} \Phi_{I} + \mathbb{Q}\beta) \cap \Phi$ if $\beta \in \Psi$,
\item $\Phi_{\beta}^{+} = \Phi_{\beta} \cap \Phi^{+}$.

\end{itemize}

We now state a condition on $\lambda$ and $I$.

\noindent \emph{Condition (*)}: for all $\beta \in \Psi_{\lambda}^{+}$, there is a $\gamma \in \Phi_{\beta}$ with $\langle \lambda + \rho, \gamma^{\lor} \rangle = 0$ and $s_{\beta} \gamma \in \Phi_{I}$.

\begin{mythm}
If condition (*) holds, then $M_{I}(\lambda)$ is irreducible.
\end{mythm}

\begin{proof}
Proved in \cite[Corollary 3]{Jantzen}, stated in \cite[Corollary 9.13]{Humphreys}.
\end{proof}

\noindent The following condition, which is stronger than condition (*), will also be useful.

\noindent \emph{Condition (**)}: $\Psi^{+}_{\lambda} = \emptyset$, that is, for any $\beta \in \Phi^{+}$ such that $\langle \lambda + \rho, \beta^{\lor} \rangle \in \mathbb{N}$ we have $\beta \in \Phi_{I}$.

We can now prove the main result of this section.

\begin{mythm}
\label{sl3thm}
Suppose $\mathfrak{g} = \mathfrak{sl}_{3}(\Of)$, and $G$ is an $F$-uniform group with $L_{G} = p^{n+1} \mathfrak{g}$. Suppose $\lambda$ is a weight such that $\lambda(p^{n} \mathfrak{h}_{R}) \subseteq R$ and that $\lambda$ is not dominant integral (i.e. such that $L(\lambda)$ is infinite-dimensional over $K$). Then $\widehat{L(\lambda)}$ is a faithful $KG$-module.
\end{mythm}

\begin{mydef}
We say that a weight $\lambda \in \mathfrak{h}_{K}^{*}$ is \emph{singular} if $\langle \lambda + \rho, \alpha^{\lor} \rangle = 0$ for some $\alpha \in \Phi^{+}$. If $\lambda$ is not singular, we say that $\lambda$ is \emph{regular}.
\end{mydef}

\begin{proof}[Proof of Theorem \ref{sl3thm}]
The proof will proceed case-by-case. In each case, we write $\Delta = \{ \alpha, \beta \}$ and write $\omega_{1}, \omega_{2}$ for the fundamental weights corresponding to $\alpha, \beta$ respectively. We know from Theorem \ref{mainthm} that if $L(\lambda) = M_{I}(\lambda)$ for some $I \subseteq \{ \alpha, \beta \}$ then the result is true for $\lambda$.

\medskip

\noindent \emph{Case 1}: $\lambda$ is a singular weight.

\medskip

\noindent Say $\lambda = a \omega_{1} + b \omega_{2}$, where $a,b \in \pi^{-n}R$. Since $\lambda$ is singular, we know that either $a + 1= 0, b +1 = 0$ or $a+b+2 = 0$.

First, suppose $b=-1$, so $\lambda = a \omega_{1} - \omega_{2}$. If $a \notin \mathbb{N}_{0}$, then $\lambda$ and $I = \emptyset$ satisfy condition (**), meaning $L(\lambda) = M(\lambda)$ and the result holds for $\lambda$. Otherwise, if $a \in \mathbb{N}_{0}$ then we can see that condition (*) holds for $I = \{ \alpha \}$ and $\lambda$. This is because $\Psi_{\lambda}^{+} = \{ \alpha + \beta \}$ and $\beta \in \Phi_{\alpha + \beta}$ satisfies $\langle \lambda + \rho, \beta^{\lor} \rangle = 0$, $s_{\alpha + \beta} \beta = -\alpha \in \Phi_{I}$. Thus $L(\lambda) = M_{I}(\lambda)$ and the result follows.

The case $a = -1$ now follows by symmetry. So suppose $a + b +2 = 0$. We cannot have both $a \in \mathbb{N}_{0}$ and $b \in \mathbb{N}_{0}$, so assume without loss of generality that $a \notin \mathbb{N}_{0}$. Then $\Ann_{\Ug} \widehat{L(\lambda)} \subseteq \Ann_{\Ug} \widehat{L(s_{\alpha} \cdot \lambda)}$ by Proposition \ref{reflectionprop}. Moreover, $s_{\alpha} \cdot \lambda = -(a+2) \omega_{1} + (a+b+1) \omega_{2} = -(a+2) \omega_{1} -\omega_{2}$, so we already know $\Ann_{KG} \widehat{L(s_{\alpha} \cdot \lambda)} = 0$ and hence $\Ann_{KG} \widehat{L(\lambda)} = 0$.

\medskip

\noindent \emph{Case 2}: $\lambda$ is a regular, non-integral weight.

\medskip

\noindent Say $\lambda = a \omega_{1} + b \omega_{2}$. First, suppose $a \in \mathbb{N}_{0}$. Then $b \notin \mathbb{Z}$. So condition (*) is met with $I = \{\alpha \}$, and $L(\lambda) = M_{\{\alpha \}}(\lambda)$. Thus $\Ann_{KG} \widehat{L(\lambda)} = 0$.

Now, suppose $a \in \mathbb{Z}_{<0}$ (so $a \leq -2$ since $\lambda$ is regular). Then $\Ann_{\Ug} \widehat{L(\lambda)} \subseteq \Ann_{\Ug} \widehat{L(s_{\alpha} \cdot \lambda)}$ by Proposition \ref{reflectionprop}, and $s_{\alpha} \cdot \lambda = -(a+2)\omega_{1} + (a+b+1) \omega_{2}$ is of the form covered above.

So the result holds for $\lambda$ if $a \in \mathbb{Z}$, and also holds if $b \in \mathbb{Z}$ by symmetry.

Now suppose $a,b \notin \mathbb{Z}$. If $a + b + 2 \notin \mathbb{N}$, then $L(\lambda) = M(\lambda)$ since condition (**) holds with $I = \emptyset$. So assume $a + b +2 \in \mathbb{N}$. Then $\Ann_{\Ug} \widehat{L(\lambda)} \subseteq \Ann_{\Ug} \widehat{L(s_{\alpha} \cdot \lambda)}$ and $s_{\alpha} \cdot \lambda = -(a+2) \omega_{1} + (a+b+1) \omega_{2}$ with $a+b+1 \in \mathbb{Z}$. So by the cases above, $\Ann_{KG} \widehat{L(s_{\alpha} \cdot \lambda)} = 0$ and the result holds for $\lambda$.

\medskip

\noindent \emph{Case 3}: $\lambda$ is a regular, integral weight.

\medskip

\noindent First, suppose $\lambda = s_{\beta} \cdot \mu$ where $\mu$ is a dominant integral weight. By \cite[Example 9.5]{Humphreys}, $M_{\{\alpha \}}(\mu)$ has composition factors $L(\mu)$ and $L(\lambda)$.  So the maximal submodule of $M_{\{\alpha \}}(\mu)$ must be isomorphic to $L(\lambda)$, and $M_{\{\alpha \}}(\mu) / L(\lambda) \cong L(\mu)$.

Now, $\widehat{M_{\{\alpha \}}(\mu)}$ is  faithful over $KG$ by Theorem \ref{mainthm}. Moreover, $\widehat{L(\lambda)} \subseteq \widehat{M_{\{\alpha \}}(\mu)}$ with $\widehat{M_{\{\alpha \}}(\mu)} / \widehat{L(\lambda)} \cong \widehat{L(\mu)}$. Since $\mu$ is dominant integral, $L(\mu)$ is finite-dimensional. Therefore $\Ann_{KG} \widehat{L(\mu)} \neq 0$, for instance a sufficiently large power of $\exp(p^{n+1}e_{\alpha}) -1$ will annihilate $\widehat{L(\mu)}$.

Fix $0 \neq y \in \Ann_{KG} \widehat{L(\mu)}$ So $y \cdot \widehat{M_{\{\alpha \}}(\mu)} \subseteq \widehat{L(\lambda)}$. Suppose $x \in \Ann_{KG} \widehat{ L(\lambda)}$. Then $x y \cdot \widehat{M_{\{\alpha \}}(\mu)} = 0$. So $xy = 0$, and thus $x = 0$ since $KG$ is a domain. So $\Ann_{KG} \widehat{L(\lambda)} = 0$.

Now, in general $\lambda$ will be linked to a dominant integral weight. So $\lambda$ has one of the forms: $s_{\alpha} \cdot \mu, s_{\beta} \cdot \mu, s_{\alpha} s_{\beta} \cdot \mu, s_{\beta} s_{\alpha} \cdot \mu, s_{\alpha} s_{\beta} s_{\alpha} \cdot \mu$.
Using Proposition \ref{reflectionprop} we have $\Ann_{KG} \widehat{L(s_{\alpha} s_{\beta} s_{\alpha} \cdot \mu)} \subseteq \Ann_{KG} \widehat{L(s_{\beta} s_{\alpha} \cdot \mu)} \subseteq \Ann_{KG} \widehat{L( s_{\alpha} \cdot \mu)} = 0$. By symmetry, also $\Ann_{KG} \widehat{L(s_{\alpha} s_{\beta} \cdot \mu)} = \Ann_{KG} \widehat{L( s_{\beta}  \cdot \mu)} = 0$.

\end{proof}

\begin{proof}[Proof of theorem \ref{theorem2}]
Suppose $M$ is an infinite-dimensional highest-weight module for $\ug$ with highest-weight $\lambda$, where $\lambda(p^{n} \mathfrak{h}_{R}) \subseteq R$. Then $M$ is a quotient of the  Verma module $M(\lambda)$.

The composition factors of $M(\lambda)$ all have the form $L(\mu)$ for some weights $\mu \in \lambda - \mathbb{N}_{0} \Delta$, which also satisfy $\mu(p^{n} \mathfrak{h}_{R}) \subseteq R$. Therefore all composition factors of $M$ also have this form.

Since $M$ is infinite-dimensional, $M$ has a composition factor of the form $L(\mu)$ which is infinite-dimensional. Then by exactness of the completion functor (Proposition \ref{flatnessprop}), $\hat{M}$ has a subquotient of the form $\widehat{L(\mu)}$. By Theorem \ref{sl3thm} this composition factor is faithful over $KG$, therefore $\hat{M}$ is faithful over $KG$.
\end{proof}

\section{Application to prime ideals of Iwasawa algebras}
\label{applicationsection}
For this section we assume $\mathfrak{g}$ is an $\Of$-lattice of a split simple $\mathbb{Z}_{p}$-Lie algebra with root system $\Phi$, and a fixed choice of root basis $\Delta$, such that $\mathfrak{g}$ has a Chevalley basis and a Cartan subalgebra $\mathfrak{h}$. We make the assumption that $R$ has uniformiser $p$ (since this is assumed in \cite[\S 10]{irredreps}). Assume $G$ is a uniform pro-p group with $L_{G} = p^{n+1} \mathfrak{g}$, that $p$ does not divide the determinant of the Cartan matrix of any root subsystem of $\Phi$, and $p$ is a very good prime for $\Phi$ in the sense of \cite[\S 6.8]{irredreps}. (In particular, for any given $\Phi$, all sufficiently large primes meet these conditions, while for type $A_{2}$ these conditions are met precisely when $p>3$).

We also take $\mathbf{G}$ to be a split simple, connected, simply connected affine algebraic group over $\Of$ such that $G \leq \mathbf{G}(\mathbb{Q}_{p})$ is an open subgroup and the Lie algebra associated to $\mathbf{G}$ is $\mathfrak{g}$. For instance if $\mathfrak{g} = \mathfrak{sl}_{3}(\mathbb{Z}_{p})$, we take $\mathbf{G} = \mathbf{SL_{3}}$ and $G = \ker ( SL_{3}(\mathbb{Z}_{p}) \rightarrow SL_{3}(\mathbb{Z}_{p} / p \mathbb{Z}_{p}))$.

\begin{mythm}
\label{application}
Suppose we can give a positive answer to question B in the case of $ \mathfrak{g}$, for any deformation parameter $n$ and for any choice of the field $K$.

Then any prime ideal of $KG$ is the annihilator of an irreducible finite-dimensional module.
\end{mythm}

\noindent Note that this theorem together with theorem \ref{theorem2} will imply theorem \ref{theorem3}.

\noindent We start with the following lemma, which is a general result in ring theory.

\begin{mylem}
Suppose $A$ is a unital $K$-algebra and $P$ is a non-zero prime ideal of $A$ of finite codimension in $A$. Then $P$ is the annihilator of an irreducible finite-dimensional $A$-module.
\end{mylem}

\begin{proof}
Let $M = A / P$, an $A$-module which is finite-dimensional over $K$ and which satisfies $\Ann_{A} M = P$. Then $M$ has a composition series, say $0 = M_{0} < M_{1} < \dots < M_{r} = M$. Let $L_{j} = M_{j} / M_{j-1}$ for $1 \leq j \leq r$, so each $L_{j}$ is an irreducible finite-dimensional $A$-module. Let $I_{j} = \Ann_{A} L_{j}$ for each $j$. So $P \subseteq I_{j}$ for each $j$.

Then $I_{1} \dots I_{r} \subseteq \Ann_{A} M = P$, so since $P$ is prime, we have $I_{j} \subseteq P$ for some $j$, thus $I_{j} =P$, and the result follows.

\end{proof}

\noindent Suppose $P$ is a non-zero prime ideal of $KG$. If $P$ has finite codimension in $KG$, it follows from the lemma above that $P$ is the annihilator of an irreducible finite-dimensional module. So now assume $KG / P$ is infinite-dimensional, we then wish to show that $P = 0$. We write $P_{m} = KG^{p^{m}} \cap P$ for any $m$.

\begin{myprop}
For sufficiently large $m \in \mathbb{N}_{0}$, the $\widehat{U(\mathfrak{g})_{mn,K}}$-module $M = (\widehat{U(\mathfrak{g})_{mn,K}} \ast G/G^{p^{m}}) \otimes_{KG} KG/P$ is infinite-dimensional. Moreover, $P_{m} \subseteq \Ann_{KG^{p^{m}}}M$ and $M$ is a cyclic $\widehat{U(\mathfrak{g})_{mn,K}}$-module.
\end{myprop}

\begin{proof}

We use \cite[\S 10.13]{irredreps}. From the lemma in \S 10.13, the canonical dimension $d(M)$ of a finitely generated $KG$-module (defined in \cite[\S 2.5]{irredreps}) satisfies $d(M) = 0$ if and only if $M$ is finite-dimensional over $K$. Now applying the proof of \cite[Theorem 10.13]{irredreps} we can see that for sufficiently large $m$, the module $M = (\widehat{U(\mathfrak{g})_{mn,K}}  \ast G/G^{p^{m}}) \otimes_{KG} KG/P$ satisfies $d(M) = d(KG / P)$. In particular, $M$ is infinite-dimensional since $KG / P$ is infinite-dimensional.

Now we want to show that $P_{m}$ annihilates $M$. By \cite[Proposition 10.6(b)]{irredreps}, we can see that $\widehat{U(\mathfrak{g})_{mn,K}}\ast G/G^{p^{m}} \cong \widehat{U(\mathfrak{g})_{mn,K}} \otimes_{RG^{p^{m}}} RG = \widehat{U(\mathfrak{g})_{mn,K}} \otimes_{KG^{p^{m}}} KG$. Therefore $M \cong (\widehat{U(\mathfrak{g})_{mn,K}} \otimes_{KG^{p^{m}}} KG) \otimes_{KG} KG / P \cong \widehat{U(\mathfrak{g})_{mn,K}} \otimes_{KG^{p^{m}}} KG / P$ as $\widehat{U(\mathfrak{g})_{mn,K}}$-modules.

Now consider the natural surjection $\widehat{U(\mathfrak{g})_{mn,K}} \rightarrow M$ of $\widehat{U(\mathfrak{g})_{mn,K}}$-modules. This induces a filtration of $M$ with respect to the filtration $\Gamma_{i} \widehat{U(\mathfrak{g})_{mn,K}}$ on $\widehat{U(\mathfrak{g})_{mn,K}}$ (as defined in \S \ref{completions}). We know $KG^{p^{m}}$ is a dense subring of $\widehat{U(\mathfrak{g})_{mn,K}}$ with respect to the filtration $\Gamma$, and therefore the image of $KG^{p^{m}}$ in $M$ is a dense subset of $M$.

Now $P_{m}$ annihilates the image of $KG^{p^{m}}$ in $M = \widehat{U(\mathfrak{g})_{mn,K}} \otimes_{KG^{p^{m}}} KG/P$, and therefore by the density annihilates $M$.

Finally, since there is a surjection $\widehat{U(\mathfrak{g})_{mn,K}} \rightarrow M$ of $\widehat{U(\mathfrak{g})_{mn,K}}$-modules, $M$ is a cyclic module as required.

\end{proof}

\noindent Now we have this result, fix some $m$ such that the $\widehat{U(\mathfrak{g})_{mn,K}} $-module $M = (\widehat{U(\mathfrak{g})_{mn,K}} \ast G/G^{p^{m}}) \otimes_{KG} KG/P$ is infinite-dimensional. 

\begin{myprop}
There is an infinite-dimensional $\widehat{U(\mathfrak{g})_{mn,K}} $-quotient $M'$ of $M$ such that the action of the centre $Z(\widehat{U(\mathfrak{g})_{mn,K}} )$ on $M'$ is locally finite.
\end{myprop}

\begin{proof}
By applying \cite[Proposition 9.4]{irredreps}, there is a quotient $M'$ of $M$ such that $d(M') > 0$ (which means $M'$ is infinite-dimensional) and such that $M'$ is locally finite over $\widehat{U(\mathfrak{g})_{nm,K}^{\mathbf{G}}}$. Now, by \cite[Theorem C]{verma}, $Z(\widehat{U(\mathfrak{g})_{mn,K}}) = \widehat{U(\mathfrak{g})_{nm,K}^{\mathbf{G}}}$ and the result follows.
\end{proof}

\noindent So now take $M'$ as in the proposition above. Since $P_{m}$ annihilates $M$, we can see that $P_{m} $ also annihilates $M'$. Also $M$ is a cyclic $\widehat{U(\mathfrak{g})_{mn,K}}$-module, so is $M'$.

\begin{myprop}
There is a finite field extension $K'$ of $K$, with ramification index $e$, uniformiser $\pi'$ and ring of integers $R'$, together with a weight $\lambda \in \mathfrak{h}_{K'}$ such that $\lambda(\pi^{nme}\mathfrak{h}_{R'}) \subseteq R'$, and an infinite-dimensional cyclic $\widehat{U(\mathfrak{g})_{nme,K'}}$-module $N$ with central character $\chi_{\lambda}$ such that $P_{me}$ annihilates $N$.
\end{myprop} 

\begin{proof}
The module $N$ is constructed in the proof of \cite[Theorem 9.5]{irredreps}. $N$ is obtained by first taking a quotient of $M'$, and then applying a base change. This means that since $P_{m}$ annihilates $M'$, $P_{me}$ annihilates $N$. Also since $M'$ is cyclic, so is $N$.
\end{proof}

\noindent Now take $K'$, $N$ and $\lambda$ as in the proposition. Write $n' = nme$. We want to use \cite[Theorem B]{Ioan}, which applies to primitive ideals with central character. So to this end we need to show that $N$ has an infinite-dimensional irreducible subquotient.

\begin{mylem}
$\widehat{U(\mathfrak{g})_{n',K'}}$ is a (both left and right) Noetherian ring.
\end{mylem}

\begin{proof}
In \cite[5.16]{irredreps} it is observed that $\widehat{U(\mathfrak{g})_{n',K'}}$ is an almost commutative affinoid $K'$-algebra and in \cite[3.8]{irredreps} it is observed that almost commutative affinoid $K'$-algebras are Noetherian.
\end{proof}

\begin{mylem}
\label{cyclicmodulelem}
Suppose $L$ is a non-zero finite-dimensional $U(\mathfrak{g}_{K'})$-module and $d \in \mathbb{N}$ with $d > \dim_{K'} L$. Then $L^{d}$ cannot be a cyclic $U(\mathfrak{g}_{K'})$-module.
\end{mylem}

\begin{proof}
Suppose $(x_{1},\dots,x_{d}) \in L^{d}$. Then $x_{1},\dots,x_{d} \in L$ are $K'$-linearly dependent since $d > \dim_{K'} L$, so choose $a_{1},\dots,a_{d} \in K'$ not all zero such that $\sum a_{i} x_{i} = 0$. Then for any $y \in U(\mathfrak{g}_{K'})$ we have $\sum a_{i} y x_{i} = 0$.

Choose any $0 \neq v \in L$ and any $i$ with $a_{i} \neq 0$. Then the element of $L^{d}$ with $i$-th entry $v$ and all other entries $0$ cannot be contained in $U(\mathfrak{g}_{K'})(x_{1},\dots,x_{d})$.
\end{proof}

\begin{mylem}
\label{boundeddimensionlem}
There exists $D \in \mathbb{N}$ such that all finite-dimensional quotients of $N$ have dimension at most $D$.
\end{mylem}

\noindent Write $W$ for the Weyl group of $\Phi$, and $\rho = \frac{1}{2} \sum_{\alpha \in \Phi^{+}} \alpha$. Then given a weight $\mu \in \mathfrak{h}_{K'}^{*}$, recall the dot action of $w \in W$ on $\mu$ is given by $w \cdot \mu = w(\mu + \rho) - \rho$.

We say a weight $\mu \in \mathfrak{h}_{K'}^{*}$ is dominant integral if $\mu(h_{\alpha}) \in \mathbb{N}_{0}$ for all $\alpha \in \Delta$. It is a standard fact in Lie algebra representation theory (see for instance \cite[7.2.6]{Dixmier}) that the irreducible highest-weight module $L(\mu)$ with highest-weight $\mu$ is finite-dimensional if and only if $\mu$ is dominant integral.

\begin{proof}[Proof of Lemma \ref{boundeddimensionlem}]
Suppose $N'$ is a finite-dimensional quotient of $N$. Then since $N$ is a cyclic $\widehat{U(\mathfrak{g})_{n',K'}}$-module and has central character $\chi$, it follows that $N'$ is also a cyclic $\widehat{U(\mathfrak{g})_{n',K'}}$-module and has central character $\chi$.

Note that any generator for $N'$ over $\widehat{U(\mathfrak{g})_{n',K'}}$ also generates $N'$ over $U(\mathfrak{g}_{K'})$: if $v \in N'$ such that $N' = \widehat{U(\mathfrak{g})_{n',K'}}v$, then $U(\mathfrak{g}_{K'})v$ is a $\widehat{U(\mathfrak{g})_{n',K'}}$-submodule of $N'$ by Proposition \ref{locallyfiniteprop} and hence $U(\mathfrak{g}_{K'})v = \widehat{U(\mathfrak{g})_{n',K'}} v = N'$.

Now by Weyl's Theorem (see for instance \cite[Theorem 7.8.11]{Homology}), $N'$ is a direct sum of simple modules. Moreover, by \cite[7.2.2]{Dixmier} any simple finite-dimensional $U(\mathfrak{g}_{K'})$-module has the form $L(\mu)$ for some weight $\mu$. So $N'$ is a direct sum of modules of the form $L(\mu)$ with $\mu \in \mathfrak{h}_{K'}^{*}$, and with $\mu$ dominant integral.

 By \cite[Theorem 1.10(b)]{Humphreys}, for any two weights $\mu_{1}$ and $\mu_{2}$ we have $\chi_{\mu_{1}} = \chi_{\mu_{2}}$ if and only if $\mu_{1} = w \cdot \mu_{2}$ for some $w \in W$.
 
  Moreover, there is at most one weight $\mu$ such that $\mu = w \cdot \lambda$ for some $w \in W$ and $\mu$ is dominant integral. To see this, suppose $\nu$ is a dominant integral weight and $w \in W$ such that $w \cdot \nu$ is also dominant integral. Then for any $\alpha \in \Phi$ we have $\langle w \cdot \nu + \rho, \alpha^{\lor} \rangle = \langle \nu + \rho, w^{-1}(\alpha)^{\lor} \rangle \in \mathbb{N}$ and hence $w^{-1}(\alpha) \in \Phi^{+}$. So $w^{-1}(\Phi^{+}) = \Phi^{+}$, which implies $w^{-1} = 1$.
 
  Thus there is at most one weight $\mu$ such that $\chi_{\mu} = \chi$ and $L(\mu)$ is finite-dimensional. If no such $\mu$ exists than the only finite-dimensional quotient of $N$ is $0$, and we are done. So assume $\mu$ satisfies $\chi_{\mu} = \chi$ and $L(\mu)$ is finite-dimensional, and let $d = \dim_{K'} L(\mu)$.

Then $N'$ is a direct sum of copies of $L(\mu)$. We know $N'$ is a cyclic $U(\mathfrak{g}_{K'})$-module, so by Lemma \ref{cyclicmodulelem} $N'$ is a copy of at most $d$ copies of $L(\mu)$. Thus $\dim_{K'}N' \leq d^{2}$, and we are done.

\end{proof}

\begin{mycor}
$N$ has an infinite-dimensional irreducible subquotient.
\end{mycor}

\begin{proof}
Since $\widehat{U(\mathfrak{g})_{n',K'}}$ is a Noetherian ring and $N$ is a finitely generated $\widehat{U(\mathfrak{g})_{n',K'}}$-module, it follows that $N$ is a Noetherian  $\widehat{U(\mathfrak{g})_{n',K'}}$-module. 

Now by Lemma \ref{boundeddimensionlem}, choose $N'$ to be a finite-dimensional quotient of $N$ of maximal dimension. Say $N' = N / N_{1}$, where $N_{1} \leq N$ is an infinite-dimensional submodule. Then $N_{1}$ is a Noetherian  $\widehat{U(\mathfrak{g})_{n',K'}}$-module since $N$ is Noetherian.

Now since $N_{1}$ is Noetherian we can choose a maximal submodule $N_{2}$ of $N_{1}$. Then $N_{1} / N_{2}$ is infinite-dimensional since otherwise $N / N_{2}$ would be a finite-dimensional quotient of $N$ with larger dimension than $N'$. So $N_{1} / N_{2}$ is an infinite-dimensional irreducible subquotient of $N$.
\end{proof}

\noindent Now let $N_{1}$ be an infinite-dimensional irreducible subquotient of $N$. Then $P_{me} \subseteq KG^{p^{me}} \subseteq K'G^{p^{me}}$ annihilates $N_{1}$, and $\Ann_{\widehat{U(\mathfrak{g})_{n',K'}}}N_{1}$ is a primitive ideal of $\widehat{U(\mathfrak{g})_{n',K'}}$ with central character $\chi_{\lambda}$. 

\begin{myprop}
There exists a weight $\mu \in \mathfrak{h}_{K'}$ with $\mu( p^{n'} \mathfrak{h}_{R'} )\subseteq R'$ such that $L(\mu)$ is infinite-dimensional and $\Ann_{\widehat{U(\mathfrak{g})_{n',K'}}} N_{1} = \Ann_{\widehat{U(\mathfrak{g})_{nme,K'}}} \widehat{L(\mu)}$.
\end{myprop}

\begin{proof}
It follows from \cite[Theorem B]{Ioan} (applied to the group scheme $\mathbf{G}(R')$) that there is a weight $\mu$ such that $\Ann_{\widehat{U(\mathfrak{g})_{n',K'}}} N_{1} = \Ann_{\widehat{U(\mathfrak{g})_{n',K'}}} \widehat{L(\mu)}$. Moreover, since $N_{1}$ is infinite-dimensional, it follows that $\Ann_{\widehat{U(\mathfrak{g})_{n',K'}}} N_{1}$ must have infinite codimension in $\widehat{U(\mathfrak{g})_{n',K'}}$. Therefore $\Ann_{\widehat{U(\mathfrak{g})_{n',K'}}} \widehat{L(\mu)}$ also has infinite codimension, and so $\widehat{L(\mu)}$ must be infinite-dimensional. Therefore $L(\mu)$ is infinite-dimensional.
\end{proof}

\noindent We can now conclude the proof of Theorem \ref{application}: given a positive answer to question B, $\widehat{L(\mu)}$ is faithful as a $K'G^{p^{me}}$-module. We can then see that $P_{me} = 0$ since $P_{me} \subseteq K'G^{p^{me}}$ annihilates $N_{1}$ and hence annihilates $\widehat{L(\mu)}$. It then follows from the argument in \cite[\S 5.7]{verma} that $P = 0$.


\begin{thebibliography}{99}

\bibitem{verma}Konstantin Ardakov, Simon Wadsley,Verma modules for Iwasawa algebras are faithful, 2013

\bibitem{irredreps}Konstantin Ardakov, Simon Wadsley, On irreducible representations of compact p-adic analytic groups,  Annals of Mathematics volume 178, issue 2, p. 453-557, 2013

\bibitem{controller}Konstantin Ardakov, The Controller Subgroup of One-sided ideals in Completed Group Rings.
Contemporary Mathematics, vol 562 (2012).

\bibitem{Dmodules} Pierre Berthelot, ${\mathcal {D}}$-modules arithmétiques. I. Opérateurs différentiels de niveau fini. Annales scientifiques de l'École Normale Supérieure, Série 4, Tome 29 (1996) no. 2, pp. 185-272. doi : 10.24033/asens.1739. 

\bibitem{Dixmier}Jacques Dixmier, Enveloping Algebras, North-Holland Mathematical Library, v.14, 1977

\bibitem{prop}J.D. Dixon, M.P.F. Du Sautoy, A. Mann, D. Segal, Analytic Pro-p Groups (2nd edition),  Cambridge University Press, 1999

\bibitem{Humphreys}J.E. Humphreys, Representations of Semisimple Lie Algebras in the BGG Category $\mathcal{O}$, Graduate Studies in Mathematics 94, 2008 

\bibitem{otherHumphreys}J.E. Humphreys, Introduction to Lie Algebras and Representation Theory, 1972

\bibitem{Jantzen}J.C. Jantzen, Kontravariante Formen auf induzierten Darstellungen halbeinfacher Lie-Algebren. Math. Ann. 226, 53–65 (1977)

\bibitem{nilpotent}Adam Jones, Primitive ideals in rational, nilpotent Iwasawa algebras, Advances in Mathematics, 403 (2022)

\bibitem{Application1}Mahesh Kakde, Proof of the main conjecture of noncommutative Iwasawa theory for totally real number fields in certain cases, Journal of Algebraic Geometry 20 (2008), 631-683

\bibitem{Application2}Debanjana Kundu, Antonio Lei, and Anwesh Ray, Arithmetic Statistics and noncommutative Iwasawa Theory, arXiv preprint arXiv:2109.13330, 2021

\bibitem{Orlik}Sascha Orlik and Matthias Strauch. On the irreducibility of locally analytic principal series
representations. Represent. Theory, 14:713–746, 2010


\bibitem{Schneider}Peter Schneider, P-adic Lie Groups, Springer, 2011

\bibitem{Schneider2}Peter Schneider and Jeremy Teitelbaum, Banach space representations and Iwasawa theory, Israel journal of mathematics 127, no. 1 (2002), 359-380

\bibitem{Ioan}Ioan Stanciu, Affinoid Enveloping Algebras and their Representations, thesis, 2020

\bibitem{Homology}Charles Weibel, An Introduction to Homological Algebra, Cambridge Studies in Advanced Mathematics, Cambridge University Press, 1994








\end{thebibliography}
\end{document}